\documentclass[leqno]{article}

\usepackage[frenchb,english]{babel}
\usepackage[T1]{fontenc}

\usepackage{amsmath}
\usepackage{amssymb}
\usepackage{amsfonts}
\usepackage{enumerate}
\usepackage{vmargin}
\usepackage[all]{xy}
\usepackage{mathrsfs}
\usepackage{mathtools}
\usepackage{lmodern}
\usepackage[colorlinks=true,pagebackref=true]{hyperref}
\usepackage{comment}
\setmarginsrb{3cm}{3cm}{3.5cm}{3cm}{0cm}{0cm}{1.5cm}{3cm}

\newcommand{\E}{\ensuremath{\mathbb{E}}}

\newcommand{\B}{\mathrm{B}} 
\newcommand{\I}{\mathrm{I}} 
\newcommand{\II}{\mathrm{II}} 
\newcommand{\K}{\ensuremath{\mathbb{K}}}
\let\L\relax 
\newcommand{\L}{\mathrm{L}}
\let\O\relax
\newcommand{\O}{\ensuremath{\mathbb{O}}}

\newcommand{\M}{\mathrm{M}}

\newcommand{\R}{\ensuremath{\mathbb{R}}}

\newcommand{\C}{\ensuremath{\mathbb{C}}}

\newcommand{\W}{\mathrm{W}}
\newcommand{\kR}{\mathfrak{R}}
\newcommand{\Zc}{\mathrm{Z}}
\newcommand{\can}{\mathrm{can}}

\newcommand{\Id}{\mathrm{Id}}

\newcommand{\JW}{\mathrm{JW}}
\newcommand{\JB}{\mathrm{JB}}
\newcommand{\JBW}{\mathrm{JBW}}
\newcommand{\JC}{\mathrm{JC}}

\renewcommand{\leq}{\ensuremath{\leqslant}}
\renewcommand{\geq}{\ensuremath{\geqslant}}

\newcommand{\qed}{\hfill \vrule height6pt  width6pt depth0pt}
\newcommand{\norm}[1]{ \| #1  \|}
\newcommand{\bnorm}[1]{ \big\| #1  \big\|}

\newcommand{\co}{\colon}

\newcommand{\ot}{\otimes}
\newcommand{\ovl}{\overline}
\newcommand{\otvn}{\ovl\ot}

\newcommand{\D}{\mathrm{D}}
\newcommand{\sa}{\mathrm{sa}}

\newcommand{\Sym}{\mathrm{Sym}} 
\newcommand{\Asym}{\mathrm{Asym}} 

\let\i\relax 
\newcommand{\i}{\mathrm{i}}
\newcommand{\ov}{\overset}

\renewcommand{\d}{\mathop{}\mathopen{}\mathrm{d}} 
\newcommand{\e}{\mathrm{e}} 
\renewcommand{\d}{\mathop{}\mathopen{}\mathrm{d}} 
\let\cal\relax
\newcommand{\cal}{\mathcal}

\DeclareMathOperator{\tr}{Tr} 
\DeclareMathOperator{\Ran}{Ran} 
\let\Re\relax 
\DeclareMathOperator{\Re}{Re} 
\DeclareMathOperator{\card}{card} 
\newtheorem{thm}{Theorem}[section]
\newtheorem{defi}[thm]{Definition}
\newtheorem{prop}[thm]{Proposition}
\newtheorem{conj}[thm]{Conjecture}

\newtheorem{cor}[thm]{Corollary}
\newtheorem{lemma}[thm]{Lemma}

\newtheorem{remark}[thm]{Remark}

\newtheorem{example}[thm]{Example}

\newenvironment{proof}[1][]{\noindent {\it Proof #1} : }{\hbox{~}\qed
\smallskip
}

\usepackage{tocloft}
\numberwithin{equation}{section}
\usepackage[nottoc,notlot,notlof]{tocbibind}

\let\OLDthebibliography\thebibliography
\renewcommand\thebibliography[1]{
  \OLDthebibliography{#1}
  \setlength{\parskip}{0pt}
  \setlength{\itemsep}{0pt plus 0.3ex}
}

\begin{document}
\selectlanguage{english}
\title{\bfseries{Positive contractive projections on noncommutative $\L^p$-spaces and nonassociative $\L^p$-spaces}}
\date{}
\author{\bfseries{C\'edric Arhancet}}

\maketitle

\begin{abstract}
We continue our investigation of contractive projections on noncommutative $\L^p$-spaces where $1 < p < \infty$  started in \cite{ArR19}. We improve the results of  \cite{ArR19} and we characterize precisely the positive contractive projections on a noncommutative $\L^p$-space associated with a $\sigma$-finite von Neumann algebra. We connect this topic to the theory of $\JW^*$-algebras. More precisely, in large cases, we are able to show that the range of a positive contractive projection is isometric to a nonassociative $\L^p$-space associated to a $\JW^*$-algebra.
\end{abstract}

\makeatletter
 \renewcommand{\@makefntext}[1]{#1}
 \makeatother

{
  \hypersetup{linkcolor=blue}
\tableofcontents
}

\section{Introduction}

A classical and old topic of Banach space theory is the investigation of the structure of projections and complemented subspaces. Note that a bounded operator $P \co X \to X$ on a Banach space $X$ is a projection if $P^2=P$ and a subspace $Y$ of $X$ is contractively complemented if it is the range of a contractive linear projection. It is known that a subspace of a smooth Banach space $X$ can be the range of at most one projection of norm one, see \cite[Theorem 6]{CoS1}. 
We refer to the surveys \cite{Ran01} and \cite{Mos06} for the state of the art in this area.

Suppose $1 \leq p <\infty$. For example, the contractively complemented subspaces of a classical (=commutative) $\L^p$-space $\L^p(\Omega)$ are all isometrically isomorphic to a $\L^p$-space. This is a classical result of Ando \cite{And66} (for a finite measure space), see also \cite[Theorem 3 p.~162]{Lac74}. See \cite{Tza69}, \cite{BeL74} and references therein for the case of general measures. This work also highlights the link between contractive projections and weighted conditional expectations. Note that Douglas previously investigated the particular case $p=1$ in \cite{Dou65}. Furthermore, a subspace $Y$ of a $\L^p$-space is the range of a positive contractive projection if and only if $Y$ is order isometric to some $\L^p$-space, see \cite[Problem 5.4.1]{AbA02} and \cite[Theorem 4.10]{Ran01}.

The noncommutative analogue of the $\L^p$-space $\ell^p$ is the Schatten space $S^p\ov{\mathrm{def}}{=} \{x \co \ell^2 \to \ell^2 : \norm{x}_p \ov{\mathrm{def}}{=} \bigl(\tr(\vert x\vert^p)\bigr)^{\frac{1}{p}}<\infty\}$. On this Banach space, the situation is much more complicated. It is easy to see that the range of a contractive projection $P \co S^p \to S^p$ on the Schatten space $S^p$ is not necessarily isometric to a Schatten space if $p \not= 2$. For example if $\sigma \co S^p \to S^p$ denotes the transpose map then it is clear that the map $P \ov{\mathrm{def}}{=}\frac{1}{2}(\Id_{S^p} +\sigma) \co S^p \to S^p$ is a positive contractive projection on the subspace of symmetric matrices of $S^p$, in sharp contrast with the setting of classical $\L^p$-spaces of measure spaces.

Nevertheless, the complete description and structure of contractively complemented subspaces of $S^p$ was achieved by Arazy and Friedman in their famous and impressive memoirs \cite{ArF78} and \cite{ArF92} for any $1 \leq p \leq \infty$. Such a subspace is isometrically isomorphic to a $\ell^p$-sum of $S^p$-Cartan factors of type I-IV. Recall that Cartan factors of this type are rectangular operators spaces, spaces of antisymmetric operators, spaces of symmetric operators and complex spin factors. 

In \cite[p.~99]{ArF92}, the authors explicitly introduce the problem to determine the structure of contractively complemented subspaces of noncommutative $\L^p$-spaces. Indeed, the case $p=1$ was already solved in \cite{FrB85} by Friedman and Russo. They proved that the range of contractive projection on a noncommutative $\L^1$-space (=predual of von Neumann algebra) is isometric to the predual of a $\JW^*$-triple, that is a weak* closed subspace of the space $\B(H,K)$ of bounded operators between Hilbert spaces $H$ and $K$ which is closed under the triple product $(x,y,z) \mapsto xy^*z+zy^*x$. Moreover, in \cite{NO02}, the authors gives the description of \textit{completely} contractively complemented subspaces of noncommutative $\L^1$-spaces. Theses spaces coincide with the preduals of $\W^*$-ternary rings of operators. Using the precise description of contractively complemented subspaces of $S^p$ given by Arazy and Friedman, it is showed in \cite[Theorem 1.1]{LRR09} that any \textit{completely} 1-complemented subspaces of $S^p$ is isometric to a direct sum of spaces of the form $S^p(H,K)$, where $H$ and $K$ are Hilbert spaces.

Almost forty years after the publication of the memoir \cite{ArF92}, it is proved in \cite{ArR19} that the range of a 2-positive contractive projection on an arbitrary noncommutative $\L^p$-space is completely order and completely isometrically isomorphic to some noncommutative $\L^p$-space. This result can be used with completely positive contractive projections. Furthermore, a description of 2-positive contractive projections is provided and even new for Schatten spaces. We can see a 2-positive contractive projection as some kind weighted conditional expectation, which remembers the situation of classical $\L^p$-spaces. This entails that a 2-positive contractive projection is necessarily completely positive. The approach is unrelated to the methods of Arazy and Friedman and relies on a two-sided change of density and a lifting argument of the projection at the level $p=\infty$. The use of \textit{non-tracial} Haagerup's noncommutative $\L^p$-spaces is crucial even for the case of Schatten spaces. 

The goal of this paper is to investigate the case of positive contrative projections, which are not necessarily 2-positive. Our first result is an improvement of the main result of \cite{ArR19}. It characterizes and describe precisely the positive contractive projections on noncommutative $\L^p$-spaces. 

\begin{thm}
\label{main-th-ds-intro-bis}
Let $\cal{M}$ be a $\sigma$-finite von Neumann algebra. Suppose $1 < p<\infty$. A bounded map $P \co \L^p(\cal{M}) \to \L^p(\cal{M})$ is a positive contractive projection if and only if there exist a faithful normal state $\varphi$ on $\cal{M}$, a positive element $h \in \L^p(\cal{M},\varphi)$ with support projection $s(h)$ and a faithful normal Jordan conditional expectation $Q \co s(h)\cal{M}s(h) \to s(h)\cal{M}s(h)$ such that
\begin{enumerate} 
	\item for any $y \in \L^p(\cal{M})$ we have $P(y)=P(s(h)ys(h))$,
	\item $s(h)$ belongs to the centralizer of the state $\varphi$,
	\item for any $x \in s(h)\cal{M}s(h)$ we have
$
\tr_{\varphi}(h^p Q(x))
=\tr_{\varphi}(h^p x)$,
\item for any $x \in s(h)\cal{M}s(h)$, we have 
\begin{equation}
\label{relevement-E_h-intro}
P\big(h^{\frac{1}{2}}xh^{\frac{1}{2}}\big)
=h^{\frac{1}{2}} Q(x) h^{\frac{1}{2}}.
\end{equation}
\end{enumerate}
Moreover, in this case $h$ belongs to the range $\Ran P$ of the projection $P$ and $Q(s(h)\cal{M}s(h))$ is a $\JW^*$-subalgebra of the $\JW^*$-algebra $(s(h)\cal{M}s(h),\circ)$.
\end{thm}
Here $(x,y) \mapsto x \circ y \ov{\mathrm{def}}{=} \frac{xy+yx}{2}$ is the Jordan product which is commutative but non-associative and $\tr_\varphi$ is the trace on the Haagerup noncommutative $\L^1$-space $\L^1(\cal{M},\varphi)$ defined in \eqref{Def-tr}. 

Recall that a $\JW^*$-algebra is a weak* closed Jordan-$*$-subalgebra of a von Neumann algebra. This notion was introduced by Edwards in the paper \cite{Edw80} (see also \cite{You78}). Note that a von Neumann algebra equipped with the Jordan product is an example of $\JW^*$-algebra. It is known that the selfadjoint parts of $\JW^*$-algebras are precisely the $\JW$-algebras. In the continuity of the classical work \cite{JNW34} of Jordan, von Neumann and Wigner on Jordan algebras, the theory of $\JW$-algebras was introduced by Topping in \cite{Top65}. By the previous remark, the theory of $\JW^*$-algebras is essentially equivalent to the theory of $\JW$-algebras.  We refer to the books \cite{AlS03}, \cite{ARU97}, \cite{CGRP14}, \cite{CGRP18}, \cite{HOS84} and references therein for more information on these algebras. These are related to the study of the state spaces of $\mathrm{C}^*$-algebras, see \cite{AHOS80}, \cite{AlS78} and the survey \cite{Alf79}. Furthermore, these algebras are connected to bounded symmetric domains, see \cite[pp.~92-93]{HOS84} for a brief overview of this topic. Finally, the category of $\sigma$-finite $\JBW$-algebras is equivalent to the category of facially homogeneous self-dual cone in real Hilbert spaces, see \cite{Ioc84}. This fact can be seen as a generalization of the well-known one-to-one correspondence between $\sigma$-finite von Neumann algebras and orientable, facially homogeneous self-dual cones in complex Hilbert spaces introduced by Connes in \cite{Con74}.

A difficulty for the identification of the range of the projection $P$ is that the Jordan conditional expectation $Q$ is not necessarily selfadjoint in the sense of \cite[p.~122]{JMX06} with respect to the restriction of the state $\psi \ov{\mathrm{def}}{=} \frac{1}{\norm{h}_p^p}\tr_{\varphi}(h^p\,\cdot)$ on the von Neumann algebra $s(h)\cal{M}s(h)$, \textit{contrary} to the case of classical conditional expectations on von Neumann algebras. So at the time of writing, we have difficulties to describe the range of the induced map $Q_1 \co \L^1(s(h)\cal{M}s(h),\psi) \to \L^1(s(h)\cal{M}s(h),\psi)$ and we does not see how use Theorem \ref{main-th-ds-intro-bis} and complex interpolation for identifying the range of $P$.
 
The \textit{main} part of this paper consists precisely to investigate the fine structure of $P$ in order to determine the range of $P$. We will prove that the range of $P$ is isometrically isomorphic to a complex interpolation space of the form $(\cal{N},\cal{N}_*)_{\frac{1}{p}}$ for a suitable compatibility, where $\cal{N}$ is a $\JW^*$-algebra and where $\cal{N}_*$ is its predual, in large cases. These spaces are nonassociative $\L^p$-spaces associated with $\JW^*$-algebra, notion which we will introduce in a companion paper \cite{Arh23}.

Finally, we give related complements in another publication \cite{Arh22}. Indeed, we study the contractively decomposable projections where contractively decomposable maps are defined in \cite{ArK23} and \cite{JuR04}. Finally, we refer to \cite{HRS03}, \cite{PiX03} and \cite{RaX03} for more information on the structure of noncommutative $\L^p$-spaces.

\paragraph{Structure of the paper}
The paper is organized as follows. Section \ref{Haagerup-noncommutative} gives a  presentation of Haagerup noncommutative $\L^p$-spaces. We equally state and prove some results on Banach space geometry which are necessary for this paper. In Section \ref{Sec-Jordan}, we recall some information on Jordan algebras and we examine some contractive projections which we call Jordan conditional expectations. In Section \ref{Sec-a-local}, we give a proof of Theorem \ref{main-th-ds-intro-bis}. We closely follow the approach of \cite{ArR19}. Lemma \ref{Lemma-smooth} allows us to clarify the proof. Section \ref{Sec-a-precise} contains an analysis of the structure of the projection $P$ of Theorem \ref{main-th-ds-intro-bis}. We prove that in large cases its range is isometrically isomorphic to a nonassociative $\L^p$-space. Finally, we raise several an open problem in Section \ref{sec-open-questions} related to the content of this paper.

%

\section{Haagerup's noncommutative $\L^p$-spaces and Banach space geometry}
\label{Haagerup-noncommutative}

It is well-known that there are several equivalent constructions of noncommutative $\L^p$-spaces associated with a von Neumann algebra. In this paper, we will use Haagerup's noncommutative $\L^p$-spaces introduced in \cite{Haa79a} and presented in a more detailed way in \cite{Ter81}. We denote by $s(x)$ the support of a positive operator $x$. If $\cal{M}$ is a von Neumann algebra equipped with a normal semifinite faithful trace, then the topological $*$-algebra of all (unbounded) $\tau$-measurable operators $x$ affiliated with $\cal{M}$ is denoted by $\L^0(\cal{M},\tau)$. 



In the sequel, we fix a normal semifinite faithful weight $\varphi$ on a von Neumann algebra $\cal{M}$ acting on a Hilbert space $H$. The one-parameter modular automorphisms group associated with $\varphi$ is denoted by $\sigma^\varphi=(\sigma_t^\varphi)_{t \in \R}$ \cite[p.~92]{Tak03}. We denote by $\mathfrak{m}_\varphi^+$ the set of all positive $x \in \cal{M}$ such that $\psi(x)<\infty$ and $\mathfrak{m}_\varphi$ its complex linear span.


For $1 \leq p <\infty$, the spaces $\L^p(\cal{M})$ are constructed as spaces of measurable operators relative not to $\cal{M}$ but to some semifinite bigger von Neumann algebra, namely, the crossed product $\tilde{\cal{M}} \ov{\mathrm{def}}{=} \cal{M} \rtimes_{\sigma^\varphi} \R$ of $\cal{M}$ by one of its modular automorphisms groups, that is, the von Neumann subalgebra of $\B(\L^2(\R,H))$ generated by the operators $\pi(x)$ and $\lambda_s \ot \Id_H$, where $x \in \cal{M}$ and $s \in \R$, defined by
\begin{equation}
\label{Translations}
\big(\pi(x)\xi\big)(t) \ov{\mathrm{def}}{=} \sigma^{\varphi}_{-t}(x)(\xi(t))
\quad \text{and} \quad
(\lambda_s \ot \Id_H)(\xi(t)) \ov{\mathrm{def}}{=} \xi(t-s), \quad t \in \R, \ \xi \in \L^2(\R,H).
\end{equation} 
For any $s \in \R$, let $W(s)$ be the unitary operator on $\L^2(\R,H)$ defined by
\begin{equation}
\label{Def-W}
\big(W(s)\xi\big)(t)
\ov{\mathrm{def}}{=} \e^{-\i s t} \xi(t),\quad \xi \in \L^2(\R,H).
\end{equation}
The dual action $\widehat{\sigma} \co \R \to \B(\tilde{\mathcal{M}})$ on $\mathcal{M}$ \cite[p.~260]{Tak03} is given by
\begin{equation}
\label{Dual-action}
\widehat{\sigma}_s(x)
\ov{\mathrm{def}}{=} W(s)xW(s)^*,\quad x \in \tilde{\cal{M}},\ s \in \R.
\end{equation}
Then, by \cite[Lemma 3.6]{Haa78a} or \cite[p.~259]{Tak03}, $\pi(\cal{M})$ is the fixed subalgebra of $\tilde{\mathcal{M}}$ under the family of automorphisms $\widehat{\sigma}_s$:
\begin{equation}
\label{carac-Pi-de-M}
\pi(\cal{M})
=\big\{x \in \tilde{\cal{M}}\ : \ \widehat{\sigma}_s(x)=x \quad \text{for all }s \in \R\big\}.
\end{equation}
We identify $\cal{M}$ with the subalgebra $\pi(\cal{M})$ in $\tilde{\cal{M}}$. If $\psi$ is a normal semifinite weight on $\cal{M}$, we denote by $\widehat{\psi}$ its Takesaki's dual weight on the crossed product $\cal{M}$, see the introduction of \cite{Haa78b} for a simple definition using the theory of operator valued weights. 
This dual weight  satisfies the $\widehat{\sigma}$-invariance relation $\widehat{\psi} \circ \widehat{\sigma}=\widehat{\psi}$, see \cite[(10) p.~]{Ter81}. In fact, Haagerup introduces an operator valued weight $T \co \cal{M}^+ \to \bar{\cal{M}}^+$ with values in the extended positive part\footnote{\thefootnote. If $\cal{M}=\L^\infty(\Omega)$ then $\bar{\cal{M}}^+$ identifies to the set of equivalence classes of measurable functions $\Omega \to [0,\infty]$.} $\bar{\cal{M}}^+$ of $\cal{M}$ and formally defined by 
\begin{equation}
\label{Operator-valued}
T(x)
=\int_\R \widehat{\sigma}_s(x)\d s
\end{equation}
and shows that for a normal semifinite weight $\psi$ on $\cal{M}$, its dual weight is 
\begin{equation}
\label{Def-poids-dual}
\widehat{\psi} 
\ov{\mathrm{def}}{=} \bar \psi\circ T
\end{equation}
where $\bar\psi$ denotes the natural extension of the normal weight $\psi$ to the whole of $\bar{\cal{M}}^+$.

By \cite[p.~301]{Str81} \cite[Th.~3.7]{Haa78a} \cite[Chap.~II, Lemma 1]{Ter81}, the map $\psi \to \widehat{\psi}$ is a bijection from the set of normal semifinite weights on $\cal{M}$ onto the set of normal semifinite $\widehat{\sigma}$-invariant weights on $\tilde{\cal{M}}$.

Recall that by \cite[Lemma 5.2 and Remark p.~343]{Haa79b} and \cite[Th.~1.1 (c)]{Haa78b} the crossed product $\tilde{\mathcal{M}}$ is semifinite and there is a unique normal semifinite faithful trace $\tau=\tau_{\varphi}$ on $\tilde{\mathcal{M}}$ satisfying 
\begin{equation}
\label{trace-carac}
(\mathrm{D}\widehat{\varphi}:\mathrm{D} \tau)_t
=\lambda_t \ot \Id_H, \quad t \in \R
\end{equation}
where $(\mathrm{D}\widehat{\varphi}:\mathrm{D} \tau)_t$ denotes the Connes cocycle \cite[p.~48]{Str81} \cite[p.~111]{Tak03} of the dual weight $\widehat{\varphi}$ with respect to $\tau$. Moreover, $\tau$ satisfies the relative invariance $\tau \circ \widehat{\sigma}_s = \e^{-s}\tau$ for any $s \in \R$ by \cite[Lemma 5.2]{Haa79b}.

If $\psi$ is a normal semifinite weight on $\cal{M}$, we denote by $h_\psi$ the Pedersen-Takesaki derivative of the dual weight $\widehat{\psi}$ with respect to $\tau$ given by \cite[Theorem 4.10]{Str81}. By \cite[Corollary 4.8]{Str81}, note that the relation of $h_\psi$ with the Radon-Nikodym cocycle of $\widehat\psi$ is
\begin{equation}
\label{Radon-Nikodym-1}
(\D\widehat\psi:\mathrm{D}\tau)_t
= h_\psi^{\i t}, \quad t \in \R.
\end{equation}

By \cite[Chap.~II, Prop.~4]{Ter81}, the mapping $\psi \to h_\psi$ gives a bijective correspondence between the set of all normal semifinite weights on $\cal{M}$ and the set of positive selfadjoint operators $h$ affiliated with $\tilde{\cal{M}}$ satisfying
\begin{equation}
\label{eq:def-L1}
\widehat{\sigma}_s(h)
=\e^{-s}h, \quad s \in \R.
\end{equation}
Moreover, by \cite[Chap.~II, Cor.~6]{Ter81}, $\omega$ belongs to $\cal{M}_*^+$ if and only if $h_\omega$ belongs to $\L^0(\tilde{\cal{M}},\tau)_+$. One may extend by linearity the map $\omega \mapsto h_\omega$ to the whole of $\cal{M}_*$. The Haagerup space $\L^1(\cal{M},\varphi)$ is defined as the set $\{h_\omega : \omega \in \cal{M}_*\}$, i.e. the range of the previous map. This is a closed linear subspace of $\L^0(\tilde{\cal{M}},\tau)$, characterized by the conditions (\ref{eq:def-L1}).

By \cite[Chap.~II, Th.~7]{Ter81}, the mapping $\omega \mapsto h_\omega$, $\cal{M}_* \to \L^1(\cal{M},\varphi)$ is a linear order isomorphism which preserves the conjugation, the module, and the left and right actions of $\cal{M}$. Then $\L^1(\cal{M},\varphi)$ may be equipped with a continuous linear functional $\tr_\varphi \co \L^1(\cal{M}) \to \C$  defined by
\begin{equation}
\label{Def-tr}
\tr_\varphi(h_\omega)
\ov{\mathrm{def}}{=} \omega(1),\quad \omega \in \cal{M}_*
\end{equation}
\cite[Chap.~II, Def.~13]{Ter81}. We also use the notation $\tr$ instead of $\tr_\varphi$. A norm on $\L^1(\cal{M},\varphi)$ may be defined by $\norm{h}_1 \ov{\mathrm{def}}{=} \tr(|h|)$ for every $h \in \L^1(\cal{M},\varphi)$. By \cite[Chap.~II, Prop.~15]{Ter81}, the map $\cal{M}_* \to \L^1(\cal{M},\varphi)$, $\omega \mapsto h_\omega$ is a surjective isometry.

More generally for $1 \leq p \leq \infty$, the Haagerup $\L^p$-space $\L^p(\cal{M},\varphi)$ associated with the normal faithful semifinite weight $\varphi$ is defined \cite[Chap.~II, Def.~9]{Ter81} as the subset of the topological $*$-algebra $\L^0(\tilde{\cal{M}},\tau)$ of all (unbounded) $\tau$-measurable operators $x$ affiliated with $\cal{M}$ satisfying for any $s \in \R$ the condition
\begin{equation}
\label{Def-Haagerup}
\widehat{\sigma}_s(x)=\e^{-\frac{s}{p}}x \quad \text{if } p<\infty 
\quad \text{and} \quad \widehat{\sigma}_s(x)=x \quad \text{ if } p=\infty	
\end{equation}
where $\widehat{\sigma}_s \co \L^0(\cal{M},\tau) \to \L^0(\cal{M},\tau)$ is here the continuous $*$-automorphism obtained by a natural extension of the dual action \eqref{Dual-action} on $\cal{M}$. By \eqref{carac-Pi-de-M}, the space $\L^\infty(\cal{M},\varphi)$ coincides with $\pi(\cal{M})$ that we identify with $\cal{M}$. The spaces $\L^p(M,\varphi)$ are closed selfadjoint linear subspaces of $\L^0(\cal{M},\tau)$. They are closed under left and right multiplications by elements of $\cal{M}$. If $h=u|h|$ is the polar decomposition of $h \in \L^0(\tilde{\cal{M}},\tau)$ then by \cite[Chap.~II, Prop.~ 12]{Ter81} we have
$$
h \in \L^p(\cal{M},\varphi)
\iff 
u \in \cal{M} \text{ and } |h| \in \L^p(\cal{M},\varphi).
$$

Suppose $1 \leq p<\infty$. By \cite[Chap.~II, Prop.~12]{Ter81} and its proof, for any $h \in \L^0(\tilde{\cal{M}},\tau)_+$, we have $h^p \in \L^0(\tilde{\cal{M}},\tau)_+$. Moreover, an element $h \in \L^0(\tilde{\cal{M}},\tau)$ belongs to $\L^p(\cal{M},\varphi)$ if and only if $|h|^p$ belongs to $\L^1(\cal{M},\varphi)$. A norm on $\L^p(\cal{M},\varphi)$ is then defined by the formula 
\begin{equation}
\label{Def-norm-Lp}
\norm{h}_p
\ov{\mathrm{def}}{=} (\tr |h|^p)^{\frac{1}{p}}
\end{equation} 
if $1 \leq p < \infty$ and by $\norm{h}_\infty \ov{\mathrm{def}}{=}\norm{h}_\cal{M}$, see \cite[Chap.~II, Def.~14]{Ter81}.

\paragraph{Case of a normal faithful linear form} If $\varphi$ is a normal faithful linear form on $\cal{M}$ then by \cite[(1.13)]{HJX10} $h_\varphi$ belongs to $\L^1(\cal{M},\varphi)$ and 
\begin{equation}
\label{HJX-1.13}
\varphi(x)
=\tr(h_\varphi x)=\tr(xh_\varphi ), \quad x \in \cal{M}.
\end{equation}

\paragraph{Duality} Let $p, p^* \in [1,\infty]$ with $\frac{1}{p}+\frac{1}{p^*}=1$. By \cite[Chap.~II, Prop.~21]{Ter81}, for any $h \in \L^p(\cal{M},\varphi)$ and any $k \in \L^{p^*}(\cal{M},\varphi)$ we have $hk,kh \in \L^1(\cal{M},\varphi)$ and the tracial property $\tr(hk)=\tr(kh)$.

If $1 \leq p < \infty$, by \cite[Ch.~II, Th.~32]{Ter81} the bilinear form $\L^p(\cal{M},\varphi) \times \L^{p^*}(\cal{M},\varphi)\to \C$, $(h,k) \mapsto \tr(h k)$ defines a duality bracket between $\L^p(\cal{M},\varphi)$ and $\L^{p^*}(\cal{M},\varphi)$, for which $\L^{p^*}(\cal{M},\varphi)$ is (isometrically) the dual of $\L^p(\cal{M},\varphi)$. 


\paragraph{Change of weight} It is essentially proved in \cite[p.~59]{Ter81} 
that $\L^p(\cal{M},\varphi)$ is independent of $\varphi$ up to an isometric isomorphism preserving the order and modular structure of $\L^p(\cal{M},\varphi)$, as well as the external products and Mazur maps. In fact given two normal semifinite faithful weights $\varphi_1, \varphi_2$ on $\cal{M}$ there is a $*$-isomorphism $\kappa \co \cal{M}_1 \to \cal{M}_2$ between the crossed products $\tilde{\cal{M}}_i\ov{\mathrm{def}}{=} \cal{M} \rtimes_{\sigma^\varphi_i} \R$ preserving $\cal{M}$, as well as the dual actions and pushing the trace on $\cal{M}_1$ onto the trace on $\cal{M}_2$, that is 
\begin{align}
\label{kappa}
\pi_2 = \kappa\circ \pi_1,\quad \widehat\sigma_2\circ\kappa 
= \kappa \circ \widehat\sigma_1 
\quad \text{and} \quad
\tau_2=\tau_1 \circ \kappa^{-1}.
\end{align}
Furthermore, $\kappa$ extends naturally to a topological $*$-isomorphism $\kappa \co \L^0(\tilde{\cal{M}}_1,\tau_1) \to \L^0(\tilde{\cal{M}}_2,\tau_2)$ between the algebras of measurable operators, which restricts to isometric $*$-isomorphisms between the noncommutative $\L^p$-spaces $\L^p(\cal{M},\varphi_1)$ and $\L^p(\cal{M},\varphi_2)$, preserving the $\cal{M}$-bimodule structures.

Moreover it turns out also that for every normal semifinite faithful weight $\psi$ on $\cal{M}$, the dual weights $\hat\psi_i$ corresponds through $\kappa$, that is $\hat{\psi}_2 \circ \kappa=\hat\psi_1$. It follows that if $\omega \in \cal{M}_*$ the corresponding Pedersen-Takesaki derivatives must verify $h_{\omega,2}=\hat{\kappa}(h_{\omega,1})$. In particular if $\omega \in \cal{M}_*^+$, we have
\begin{equation}
\tr_{\varphi_1} h_{\omega,1}
\ov{\eqref{Def-tr}}{=} \omega(1) 
\ov{\eqref{Def-tr}}{=} \tr_{\varphi_2} h_{\omega,2}
=\tr_{\varphi_2} \kappa(h_{\omega,1}).
\end{equation}
Hence $\kappa \co \L^1(\cal{M},\varphi_1) \to \L^1(\cal{M},\varphi_2)$ preserves the traces:
\begin{equation}
\label{Trace-preserving}
\tr_{\varphi_1}
=\tr_{\varphi_2} \circ \kappa.
\end{equation}
Since $\kappa$ preserves the $p$-powers operations, i.e. $\kappa(h^p)=(\kappa(h))^p$ for any $h \in \L^0(\tilde{\cal{M}})$, it induces an isometry from $\L^p(\cal{M},\varphi_1)$ onto $\L^p(\cal{M},\varphi_2)$. It is not hard to see that this isometry is completely positive and completely isometric, a fact which is of first importance for our study.
 
This independence allows us to consider $\L^p(\cal{M},\varphi)$ as a particular realization of an abstract space $\L^p(\cal{M})$.


\paragraph{Centralizer of a weight} Recall that the centralizer \cite[p.~38]{Str81} of a normal semifinite faithful weight is the sub-von Neumann algebra $\cal{M}^\varphi=\{x \in \cal{M}: \sigma_t^\varphi(x)=x\text{ for all } t \in \R\}$. If $x \in \cal{M}$, we have by \cite[(2) p.~39]{Str81}
\begin{equation}
\label{Charac-centralizer}
x \in \cal{M}^\varphi
\iff x\mathfrak{m}_{\varphi} \subset \mathfrak{m}_{\varphi},\, \mathfrak{m}_{\varphi}x \subset \mathfrak{m}_{\varphi} \text{ and } \varphi(xy)=\varphi(yx) \text{ for any } y \in \mathfrak{m}_{\varphi}. 
\end{equation}

\paragraph{Reduced noncommutative $\L^p$-spaces} 


If the projection $e$ \textit{belongs} to the centralizer of $\varphi$, the restriction $\varphi_e$ of $\varphi$ on $e\cal{M}e$. It results from (\ref{Charac-centralizer}) that the weight $\varphi_e$ is still semifinite and is well-known that we can identify $\L^p(e\cal{M}e,\varphi_e)$ with the subspace $e\L^p(\cal{M},\varphi)e$ of $\L^p(\cal{M},\varphi)$ (see \cite[p.~508]{Wat98}). 
Moreover, we have the following result. 


\begin{lemma}
\label{Lemme-trace coincide}
The Haagerup trace $\tr_\varphi$ restricts to $\tr_{\varphi_e}$  on $\L^1(e\cal{M}e)$. 
\end{lemma}



Let $e \in \cal{M}$ be a projection. Let us construct a normal semifinite weight on $\cal{M}$ with centralizer containing $e$. 
Consider two normal semifinite faithful weights $\varphi_1$ and $\varphi_2$ on $e\cal{M}e$ and $e^\perp \cal{M} e^\perp$. By \cite[p.~155]{RaX03}, we can define a normal semifinite faithful weight $\varphi$ on $\cal{M}$ by 
\begin{equation}
\label{Extension-weight1}
\varphi(x)
\ov{\mathrm{def}}{=} \varphi_1(exe)+\varphi_2(e^\perp x e^\perp), \quad x \in \cal{M}_+.
\end{equation}
Moreover, $e$ belongs to the centralizer of $\varphi$ by \eqref{Charac-centralizer} and we have $\varphi_e=\varphi_1$.


%
%

The following is an easy folklore observation.

\begin{lemma}
\label{lemma2-GL}
Let $\cal{M}$ be a von Neumann algebra and  $1 \leq p <\infty$. Let $h$ be a positive element of $\L^p(\cal{M})$.
\begin{enumerate}
	\item The map $s(h)\cal{M}s(h) \to \L^p(\cal{M})$, $x \mapsto h^{\frac{1}{2}} x h^{\frac{1}{2}}$ is injective.
	\item Suppose $1 \leq p <\infty$. The subspace $h^{\frac{1}{2}} \cal{M} h^{\frac{1}{2}}$ is dense in $s(h)\L^p(\cal{M})s(h)$ for the topology of $\L^p(\cal{M})$.
\end{enumerate}
\end{lemma}

\paragraph{Lifting result}

Our main tool will be the following result of \cite{ArR19} which can be proved with the same ideas.

\begin{thm}
\label{Th-relevement-cp} 
Let $\cal{M}$ and $\cal{N}$ be von Neuman algebras. Suppose $1 \leq p < \infty$. Let $T \co \L^p(\cal{M}) \to \L^p(\cal{N})$ be a positive linear map. Let $h$ be a positive element of $\L^p(\cal{M})$. 
Then there exists a unique linear map $v \co \cal{M} \to s(T(h))\cal{N}s(T(h))$ such that 
\begin{equation}
\label{equa-relevement}
T\big(h^{\frac{1}{2}}xh^{\frac{1}{2}}\big)
=T(h)^{\frac{1}{2}}v(x)T(h)^{\frac{1}{2}},\qquad x \in \cal{M}.
\end{equation}
Moreover, this map $v$ is unital, contractive, positive and normal. Furthermore, if $T$ is $n$-copositive then $v$ is $n$-copositive.  
\end{thm}

\paragraph{Extension of maps on noncommutative $\L^p$-spaces}
Let $\cal{M}$ be a von Neumann algebra equipped with a normal semifinite faithful weight $\varphi$. Suppose that $\cal{N}$ is another von Neumann algebra equipped with a normal faithful weight $\psi$. Consider a unital positive map $T \co \cal{M} \to \cal{N}$ such that $\psi \circ T =\varphi$. Given $1 \leq p<\infty$, the map
\begin{equation}
\label{Map-extension-Lp}
\begin{array}{cccc}
T_p  \co &  h_\varphi^{\frac{1}{2p}} \mathfrak{m}_\varphi h_\varphi^{\frac{1}{2p}}  &  \longrightarrow   &  \L^p(\cal{N}) \\
           &   h_\varphi^{\frac{1}{2p}}x h_\varphi^{\frac{1}{2p}}  & \longmapsto &  h_\psi^{\frac{1}{2p}}T(x) h_\psi^{\frac{1}{2p}}  \\
\end{array}
\end{equation}
extends to a contractive map $T_p$ from $\L^p(\cal{M})$ into $\L^p(\cal{N})$. See \cite[Remark 5.6]{HJX10}.

\paragraph{Normalized duality mappings} Recall that a normed linear space $X$ is said to be strictly convex (or rotund) if for any $x,y \in X$ the equalities $\frac{\norm{x+y}_X}{2}=\norm{x}_X=\norm{y}_X$ imply $x=y$.

Let $X$ be a Banach space. For each $x \in X$, we can associate \cite[Definition 2.12]{Pat18} the subset
\begin{equation}
\label{Def-JX}
J_X(x) 
\ov{\mathrm{def}}{=} \big\{x^* \in X^*\ : \ \langle x, x^* \rangle_{X,X^*}= \norm{x}_X^2 = \norm{x^*}_{X^*}^2\big\}
\end{equation}
of the dual $X^*$. 

The multivalued operator $J_X \co X \to X^*$ is called the normalized duality mapping of $X$. From the Hahn-Banach theorem, for every $x \in X$, there exists $y^* \in X^*$ with $\norm{y^*}_{X^*} = 1$ such that $\langle x, y^*\rangle_{X,X^*} = \norm{x}_X$. Using $x^*=\norm{x}_Xy^*$, we conclude that $J_X(x)\not= \emptyset$ for each $x \in X$. If the dual space $X^*$ is strictly convex, $J_X$ is single-valued. 

When $X$ is a reflexive strictly convex Banach space with a strictly convex dual space $X^*$, $J_X$ is a singlevalued bijective map and its inverse $J_X^{-1} \co X^* \to X^{**}=X$ is equal to $J_{X^*} \co X^* \to X$. 


If the Banach space $X$ is a noncommutative $\L^p$-space, we have the following explicit description of the normalized duality mapping of \cite{ArR19}.

\begin{lemma}
\label{Lemma-J-lp}
Suppose $1<p<\infty$. If $h$ belongs to $\L^p(\cal{M})$ with polar decomposition $h=u|h|$ then we have
\begin{align}
\label{comput-J}
J_{\L^p(\cal{M})}(h)
=\norm{h}^{2-p}_{p}|h|^{p-1}u^*.
\end{align}
\end{lemma}


\paragraph{Projections} 
The following is \cite[Theorem 3.2.6 p.~297]{Meg98} combined with \cite[5.10 p.~148]{FHHMPZ01}.

\begin{prop}
\label{Prop-Fabian}
Let $X$ be a Banach space and consider a bounded map $P \co X \to X$. Then $P$ is a projection if and only if $P^*\co X^* \to X^*$ is a projection. In this case, $P(X)^*$ is isomorphic to $P^*(X^*)$.
\end{prop}

The following lemma is a generalization of \cite[Lemma 1]{And66} and is fundamental for us.

\begin{lemma}
\label{Lemma-smooth}
Let $X$ be a smooth strictly convex reflexive Banach space. Let $P \co X \to X$ be a contractive projection and $x$ be an element of $X$. Then $x$ belongs to $\Ran P$ if and only if $J_X(x)$ belongs to $\Ran P^*$.
\end{lemma}

\begin{proof}
Recall that by \cite[Theorem 3.2.6 p.~297]{Meg98} the adjoint map $P^* \co X^* \to X^*$ is a contractive projection. Note that $X^*$ is strictly convex by \cite[Proposition 5.4.7 p.~481]{Meg98}.

$\Rightarrow$: Suppose $x \in \Ran P$. We have
\begin{align*}
\MoveEqLeft
\norm{x}_{X}^2           
\ov{\eqref{Def-JX} }{=} \langle x,J_X(x) \rangle_{X,X^*}
=\big\langle P^2(x),J_X(x) \big\rangle_{X,X^*}
=\big\langle P(x),P^*(J_X(x)) \big\rangle_{X,X^*} \\
&=\big\langle x,P^*(J_X(x)) \big\rangle_{X,X^*} 
\leq \norm{x}_{X} \norm{P^*(J_X(x))}_{X^*}
\leq \norm{x}_{X} \norm{P^*(J_X(x))}_{X^*} \\
&\leq \norm{x}_{X}\norm{J_X(x)}_{X^*}
\ov{\eqref{Def-JX} }{=}\norm{x}_{X}^2.
\end{align*}
We infer that $\norm{P^*(J_X(x))}_{X^*}= \norm{x}_{X}$ and $\big\langle x,P^*(J_X(x)) \big\rangle_{X,X^*}=\norm{x}_{X}^2$. Since $X^*$ is strictly convex, we conclude that $P^*(J_X(x))=J_X(x)$, i.e. $J_X(x)$ belongs to $\Ran P^*$. 

$\Leftarrow$: Suppose that $J_X(x)$ belongs to $\Ran P^*$. Since $X$ is strictly convex, the first part applied to $P^*$ instead of $P$ shows that $x=J_{X^*}J_X(x)$ belongs to $\Ran P$.
\end{proof}

We will use the following well-known result. See \cite[Lemma 4.8]{ArK23} for a slightly more general statement.

\begin{lemma}
\label{Lemma-interpolation}
Let $(E_0,E_1)$ be an interpolation couple and let $C$ be a contractively complemented subspace of $E_0+E_1$. We assume that the corresponding contractive projection $P \co E_0+E_1 \to E_0+E_1$ satisfies $P(E_i) \subset E_i$ and that the restriction $P \co E_i \to E_i$ is contractive for $i=0,1$. Then $(E_0 \cap C,E_1 \cap C)$ is an interpolation couple and the canonical inclusion $J \co C \to E_0+E_1$ induces an isometric isomorphism $\tilde{J}$ from $(E_0 \cap C,E_1 \cap C)_\theta$ onto the subspace $P((E_0,E_1)_\theta)=(E_0,E_1)_\theta \cap C$ of $(E_0,E_1)_\theta$. 
\end{lemma}

\section{Jordan algebras and Jordan conditional expectations}
\label{Sec-Jordan}

\paragraph{Various Jordan algebras} A Jordan algebra $A$ over a field $\K$ is a vector space $A$ over $\K$ equipped with a commutative bilinear product that satisfies $(x^2 \circ y)\circ x=x^2\circ (y \circ x)$ for any $x,y \in A$, see e.g. \cite[Definition 1.1 p.~3]{AlS03}. A Jordan algebra $A$ over $\R$ is called formally real \cite[p.~69]{HOS84} if for any $x_1,\ldots,x_n\in A$ the relation $x_1^2+\cdots +x_n^2=0$ implies $x_1=\cdots=x_n=0$. Following \cite[Definition 1.5 p.~5]{AlS03}, a $\JB$-algebra is a Jordan algebra over $\R$ with identity element 1 equipped with a complete norm satisfying the properties $\norm{x \circ y} \leq \norm{x} \norm{y}$, $\norm{x^2}=\norm{x}^2$, $\norm{x^2} \leq \norm{x^2+y^2}$ for any $x,y \in A$. A $\JBW$-algebra is a $\JB$-algebra which is a dual Banach space \cite[p.~111]{HOS84}. In this case, the predual is unique.

A $\JB^*$-algebra \cite[p.~91]{HOS84} \cite[Definition 3.3.1 p.~345]{CGRP14} is a complex Banach space $A$ which is a complex Jordan algebra equipped with an involution satisfying 
\begin{equation}
\label{def-JBstar}
\norm{x \circ y} \leq \norm{x} \norm{y},\quad  \norm{x^*}=\norm{x} 
\quad \text{and} \quad \norm{\{x,x^*,x\}}=\norm{x}^3
\end{equation}
for any $x,y \in A$, where we use the Jordan triple product $\{x,y,z\}\ov{\mathrm{def}}{=} (x \circ y) \circ z+(y \circ z)\circ x-(x \circ z) \circ y$. A $\JBW^*$-algebra \cite[p.~4]{CGRP18} is a $\JB^*$-algebra which is a dual Banach space.

Let $H$ be a complex Hilbert space. A $\JC$-algebra \cite[Definition 2.1.1]{Sto13} \cite[p.~75]{HOS84} is a norm closed real linear subspace of selfadjoint operators of $\B(H)$ closed under the Jordan product $(x,y) \mapsto x \circ y \ov{\mathrm{def}}{=} \frac{1}{2}(xy+yx)$. By \cite[p.~13]{Sto13}, the selfadjoint part $A_\sa$ of a $\mathrm{C}^*$-algebra $A$ is a $\JC$-algebra.

A $\JC^*$-algebra (also called Jordan $\mathrm{C}^*$-algebra) is a norm-closed $*$-subalgebra of $(\B(H), \circ)$ \cite[p.~345]{CGRP14}. 
If $A$ is a $\JC^*$-algebra then $(A,\circ)$ is a $\JB^*$-algebra by \cite[p.~345]{CGRP14}. A $\mathrm{C}^*$-algebra $A$ is of course a $\JC^*$-algebra. 

An element $p$ of $\JB$-algebra such that $p \circ p=p$ is called a projection.

\paragraph{Centers and factors} Two elements $a$ and $b$ of a Jordan algebra $A$ are said to operator commute \cite[p.~44]{HOS84} if for any $c \in A$ we have $(a \circ c) \circ b= a \circ (c \circ b)$. The centre $\Zc(A)$ of $A$ is the set of all elements of $A$ which operator commute with other element of $A$. By \cite[Lemma 2.5.3 p.~45]{HOS84}, the centre is an associative subalgebra of $A$. Following \cite[p.~115]{HOS84}, if the centre of a $\JBW$-algebra $A$ only consists of scalar multiples of the identity, we say that $A$ is a $\JBW$-factor.

If $p$ is a projection of a $\JW$-algebra, the smallest central projection $q$ such that $q \geq p$ is called the central cover of $p$ and denoted by $c(p)$ \cite[Definition 2.38 p.~56]{AlS03}. We say that a projection $p$ of a $\JW$-algebra $A$ is abelian if the $\JW$-subalgebra $pAp$ is associative \cite[p.~122]{HOS84}. By \cite[Definition 4.24]{AlS03}, this is equivalent to $pAp$ consists of mutually commuting elements.

We refer to \cite[Theorem 6.1.40 p.~362]{CGRP18} for the classification of $\JBW$-factors.

\paragraph{$\JW$-algebras} Recall that a (concrete) $\JW$-algebra \cite[p.~95]{HOS84} \cite[p.~14]{ARU97} \cite[p.~20]{Sto13} is a weak* closed $\JC$-algebra, that is a weak* closed Jordan subalgebra of $\B(H)_\sa$, that is a real linear space of selfadjoint operators which is closed for the weak* topology and closed under the Jordan product $\circ$. Note that a $\JW$-algebra is a $\JBW$-algebra by \cite[p.~95]{HOS84}. Recall that a $\JBW$-algebra is always unital by \cite[Lemma 4.1.7]{HOS84}. By \cite[Proposition 1.49 p.~28]{AlS03}, two elements $x$ and $y$ of a $\JW$-algebra operator commute if and only if $x$ and $y$ commute.

Recall that a $\JW$-algebra $A$ is reversible \cite[Definition 4.24]{AlS03} \cite[p.~25]{HOS84} if it is closed under symmetric products, i.e. if $a_1,\ldots,a_k \in A$ then 
$$
a_1a_2 \cdots a_k+a_ka_{k-1}\cdots a_1 \in A.
$$

\begin{example} \normalfont
\label{spin-factor}
A spin system \cite[6.1.2 p.~135]{HOS84} is a set $\cal{P}$ of at least two symmetries (i.e. selfadjoint unitaries) $\not=\pm \Id$ in $\B(H)$ which satisfy $s \circ t =0$, i.e. $st=-ts$, for any $s, t \in \cal{P}$ with $s \not= t$. If $\ovl{\mathcal{P}}$ is the weak closure of the linear span of $\cal{P}$, then $S \ov{\mathrm{def}}{=} \R\Id \oplus \ovl{\cal{P}}$ is a $\JW$-algebra \cite[pp.~14-15]{ARU97}. These $\JW$-algebras are called (real) spin factors. 

It is possible to give an abstract definition \cite[Definition 3.33 p.~91]{AlS03} of spin factors. Let $H$ be a real Hilbert space of dimension at least 2, and let $\R 1$ denote a one dimensional real Hilbert space with unit vector 1. Let $A \ov{\mathrm{def}}{=}  H \oplus \R 1$ and consider the product $\circ$ on $A$ defined by
\begin{equation}
\label{product-spin-factor}
(a+\lambda 1)  \circ  (b  + \mu 1)  
\ov{\mathrm{def}}{=} \mu a  + \lambda b  +  (\langle a ,  b\rangle  +  \lambda \mu)1, \quad a,b \in H, \lambda,\mu \in \R.
\end{equation}
and the norm $\norm{a+\lambda 1}_A  \ov{\mathrm{def}}{=} \norm{a}_H  + |\lambda|$. Then $A$ is a $\JB$-algebra by \cite[Lemma 6.1.3 p.~136]{HOS84} which is isomorphic to $S$. 

By \cite[Theorem 3]{Top66} \cite[Proposition 6.1.5 p.~137]{HOS84}, two spin factors are isomorphic if and only if their real Hilbert space dimensions are equal. If $\card \cal{P}=4$ or $\card \cal{P} \geq 6$ then the spin factor is non-reversible by \cite[Lemma 2.3.2]{Sto13} \cite[Theorem 6.2.5 p.~141]{HOS84}.
\end{example}

\begin{example} \normalfont
If $\O$ is the algebra of octonions, then the space 
$$
\mathrm{H}_3(\O)
=\left\{\begin{bmatrix}
   a  & \alpha & \beta  \\
   \ovl{\alpha}  & b & \gamma	\\
    \ovl{\beta} & \ovl{\gamma} &  c  \\
\end{bmatrix}: \alpha,\beta,\gamma \in \O, a,b,c \in \R \right\}
$$ 
of hermitian 3x3 matrices with entries in $\O$ equipped with the product $(x,y) \mapsto x \circ y=\frac{1}{2}(xy+yx)$ 
is a unital formally real Jordan algebra by \cite[Proposition 2.9.2 p.~69]{HOS84} of dimension 27. By \cite[Corollary 3.1.7 p.~77]{HOS84} and its proof, we can equip $\mathrm{H}_3(\O)$ with a norm that makes it a $\JB$-algebra. With this structure, $\mathrm{H}_3(\O)$ is a $\JBW$-factor \cite[Theorem 6.1.40 p.~362]{CGRP18}. By \cite[p.~75]{HOS84}, $\mathrm{H}_3(\O)$ is not a $\JW$-algebra. Note that by \cite[Theorem 4.5]{AlS03} every $\JBW$-factor other than $\mathrm{H}_3(\O)$ is a $\JW$-algebra.
\end{example}

\paragraph{Purely exceptional $\JBW$-algebras} Following \cite[7.2.1 p.~155]{HOS84}, we say that a $\JB$-algebra $A$ is purely exceptional \cite[Theorem 7.2.3 p.~155]{HOS84} if there is no nonzero homomorphism of $A$ into a $\JC$-algebra. By \cite[Theorem 7.2.3 p.~155]{HOS84}, a $\JBW$-algebra $A$ can be uniquely decomposed as a direct sum 
\begin{equation}
\label{decompo}
A
=A_{\mathrm{sp}}\oplus A_{\textrm{exp}}
\end{equation}
where $A_{\mathrm{sp}}$ is a $\JW$-algebra and $A_{\textrm{exp}}$ is a purely exceptional $\JBW$-algebra. \textbf{SEE \cite{Shu79}}

\paragraph{Purely real $\JW$-algebras} A real von Neumann algebra \cite[p.~15]{ARU97} is a real $*$-subalgebra $R$ of $\B(H)$ which is weakly closed satisfying $R \cap \i R=\{0\}$. 
Given a $\JW$-algebra $A$, we denote\footnote{\thefootnote. We warn the reader that this algebra is sometimes denoted $\ovl{\kR(A)}$ in the literature.} by $\kR(A)$ the closure for the weak* topology of the real algebra generated by $A$ in $\B(H)$ (note that this algebra is closed under adjoints). If $A$ is a reversible $\JW$-algebra then $A=\kR(A)_\sa$ by \cite[Lemma 4.25 p.~113]{AlS03} and $A''=\kR(A)+\i \kR(A)$ by \cite[Theorem 1.5]{ARU97} and \cite[Theorem 2.4]{Sto68}.

A $\JW$-algebra $A$ is said purely real \cite[p.~15]{ARU97} if $A$ is reversible and if $\kR(A) \cap \i \kR(A)=\{0\}$. In this case,  by \cite[pp.~21-22]{ARU97} or \cite[Lemma 3.2]{Sto68}, the map $\alpha \co A'' \to A''$, $z+\i y \mapsto z^*+\i y^*$ is a $*$-antiautomorphism of order 2 and it is easy to check that
\begin{equation}
\label{}
A
=\big\{x \in (A'')_\sa : \alpha(x)=x\big\}, \quad 
A_\C=\big\{x \in A'' : \alpha(x)=x\big\}, \quad 
\kR(A)=\big\{x \in A'' : \alpha(x)=x^* \big\}
\end{equation}
where $A_\C \ov{\mathrm{def}}{=} A+\i A$ denotes the complexification of $A$. The map $P_\can \ov{\mathrm{def}}{=} \frac{\Id+\alpha}{2} \co A'' \to A''$ is a positive contractive normal unital projection called canonical projection of $A''$ onto $A_\C$.

We will use the following property \cite[Proposition 1.5.1]{ARU97}.

\begin{prop}
\label{prop-Ayu151}
A purely real $\JW$-factor $A$ is not isomorphic to the selfadjoint part of a von Neumann algebra if and only if the von Neumann algebra $A''$ is a factor.
\end{prop}

\begin{remark} \normalfont
The property of being purely real is not an invariant under isomorphisms. See \cite[p.~1427]{Ayu87}.
\end{remark}

We will use the observation of \cite[Proposition 1.1.11]{ARU97}.

\begin{prop}
\label{Prop-center}
Let $A$ be a purely real $\JW$-algebra. Then $A$ does not not admits summands isomorphic to the selfadjoint part of a von Neumann algebra if and only if $\Zc(A'')=\Zc(A)+\i\Zc(A)$.
\end{prop}

The following is \cite[Theorem 3.2]{HaS95} (see also \cite[Theorem 3.6]{Sto97}).

\begin{thm}
\label{Th-HaS95-3.2}
Let $\cal{M}$ be a von Neumann algebra and $A$ be a reversible $\JW$-subalgebra such that $\kR(A) \cap \i \kR(A) = \{0\}$, $A''=\kR(A) + \i \kR(A)$ and $\Zc(A) = \Zc(A'')_\sa$. Suppose $Q \co \cal{M} \to \cal{M}$ is a faithful normal projection on $A$. Then there exists a unique faithful normal conditional expectation $\E \co \cal{M} \to \cal{M}$ on $A''$ such that if $P_\can \co A'' \to A''$ is the canonical projection on $A$, then $Q =P_\can \circ \E$.
\end{thm}

\paragraph{Type} Following \cite[Definition 3.21 p.~86]{AlS03}, we say that a $\JW$-algebra $A$ of type $\I$ is of type $\I_n$, where $n$ is some cardinal number, if there exists an orthogonal family $p_i$ of $n$ abelian projections in $A$ such that $1=\sum_i e_i$ and $c(e_i)=1$ for any $i$. It follows from \cite[Proposition 5.3.5 p.~131]{HOS84} that any $\JBW$-algebra of type $\I$ can be uniquely decomposed into a direct sum of $\JW$-algebras of type $\I_n$. We refer to \cite[Theorem 3.39 p.~95]{AlS03} for the classification of $\JW$-factors of type $\I$.

\paragraph{Decomposition} An arbitray $\JW$-algebra $A$ can be uniquely reduced \cite[Theorem 13]{Top65} \cite[Theorem 1.1]{Ayu87} into a direct sum of five $\JW$-algebras of the following types:
\begin{enumerate}
\item modular of type $\I$ (type $\I_{\mathrm{fin}}$),
\item properly nonmodular locally modular of type $\I$ (type $\I_\infty$),
\item modular of type $\II$ (type $\II_1$),
\item properly nonmodular of type $\II$ (type $\II_\infty$),
\item purely nonmodular (type III).
\end{enumerate}
We refer to \cite{Top65} for the definitions of used notions. Moreover, each $\JW$-factor belongs to one and only one of these types. Note that if a $\JW$-algebra $A$ coincides with the selfadjoint part of a von Neumann algebra, then this decomposition agrees with the classical classification of von Neumann algebras.

Recall that by \cite[Corollary 4.30]{AlS03}, a $\JW$-algebra $A$ is reversible if and only if the $\I_2$ summand of $A$ is reversible. For the factors, we have the following result \cite[Corollary 6.5]{Sto66}.

\begin{thm}
\label{thm-factors}
A $\JW$-factor is either reversible or totally non reversible (hence of type $\I_2$).
\end{thm}

\begin{example} \normalfont
By \cite[Proposition 2.3]{ASS78} \cite[p.~477]{Sta81}, a type $\I_1$ JW-algebra $A$ is isomorphic to the Jordan algebra $\mathrm{C}(X,\R)$ of all real-valued continuous functions on a compact Hausdorff hyperstonian space $X$. \textbf{It is left to the reader use \cite{DDLS16}} that using \cite[Theorem 7.19]{BGL22} that this algebra is isomorphic to $\L^\infty_\R(\Omega)$ for some localizable measure space $\Omega$. There exists a unique factor of type $\I_1$ up to isomorphism, the factor $\R1$, see \cite[p.~96]{AlS03}. 
\end{example}

\begin{example} \normalfont
\label{ex-type-II-JW}
The type $\I_2$ $\JW$-algebras were classified by Stacey in \cite[Theorem 2]{Sta82}. If $A$ is a $\JW$-algebra with separable predual then $A$ has type $\I_2$ if and only if there exist an index set $I$, a family $(\Omega_i)_{i \in I}$ of second countable locally compact spaces, a family $(\mu_i)_{i \in I}$ of Radon measures on the spaces $\Omega_i$ and a family $(S_i)_{i \in I}$ of spin factors, each of dimension strictly greater than 1 and at most countable giving an isomorphism 
$$
A
=\oplus_{i \in I} \L^\infty_\R(\Omega_i,S_i).
$$ 
By \cite[Theorem 6.1.8 p.~138]{HOS84}, if $A$ is a $\JBW$-algebra then $A$ is a $\JBW$ factor of type $\I_2$ if and only if $A$ is isomorphic to a spin factor.
\end{example}

We will use the following result of Haagerup and St\o rmer \cite[Theorem 2.1]{HaS95}.

\begin{thm}
\label{th-proj-finite}
Let $A$ be a $\JW$-algebra of type $\I_2$ and let $A''$ be the von Neumann algebra generated by $A$. Then there exists a faithful normal projection $P \co A'' \to A''$ onto $A$ if and only if $A''$ is finite.
\end{thm}


\paragraph{$\JW^*$-algebras} A $\JW^*$-algebra is a weak* closed $\JC^*$-subalgebra of $\B(H)$, that is a weak* closed $*$-subalgebra of $(\B(H), \circ)$.  
If $M$ is a von Neumann algebra then $(M,\circ)$ is obviously a $\JW^*$-algebra. A $\JW^*$-algebra is a $\JBW^*$-algebra. The selfadjoint part of a $\JW^*$-algebra is a $\JW$-algebra. Conversely, if $A$ is a $\JW$-algebra (included in $\B(H)$) then the complexification $A_\C=A+\i A$ is a $\JW^*$-subalgebra of $\B(H)$. For useful results which can be used for transfering results from $\JW$-algebras to $\JW^*$-algebras and vice versa, we refer to \cite[pp.~4-5]{BHK17} and \cite[Corollary 5.1.29 p.~9]{CGRP18}.

We refer to \cite[Proposition 6.1.41 p.~362]{CGRP18} for the classification of $\JBW^*$-factors.

\begin{example} \normalfont
\label{Ex-von-Neumann-algebra}
A von Neumann algebra $\cal{M}$ equipped with the Jordan product
\begin{equation}
\label{Jordan-product}
x \circ y 
\ov{\mathrm{def}}{=} \frac{1}{2}(xy+yx), \quad x,y \in \cal{M} 
\end{equation}
is a $\JW^*$-algebra.
\end{example}

\begin{example} \normalfont
\label{Ex-mat-sym}
By \cite[Proposition 25.2.2 p.~513]{Isi19}, the space $\Sym_n \ov{\mathrm{def}}{=} \{x \in \M_n : x^t=x \}$ of symmetric complex matrices of $\M_n$ is a $\JW^*$-algebra (called  Cartan factor of type $\II_n$) whose the associated $\JW$-algebra is reversible.
\end{example}

\begin{example} \normalfont
If $n$ is an integer, by \cite[Proposition 25.2.2 p.~513]{Isi19} the space $\Asym_{2n} \ov{\mathrm{def}}{=} \{x \in \M_{2n} : x^t=-x \}$ of skew-symmetric complex matrices of $\M_{2n}$ is equipped with a structure of reversible $\JW^*$-algebra.
\end{example}

\begin{example} \normalfont
By \cite[Proposition 6.1.41 p.~362]{CGRP18}, the complexification $\mathrm{H}_3(\O)_\C=\mathrm{H}_3(\O_\C)$ of the $\JBW$-factor $\mathrm{H}_3(\O)$ is equipped with a structure of $\JBW^*$-factor.
\end{example}

\paragraph{Traces} A trace on a $\JBW$-algebra $A$ is a function $\tau$ on the set $A_+$ of positive
elements of $A$ with values in $[0,+\infty]$ satisfying the following conditions:
\begin{enumerate}
\item $\tau(x+y)=\tau(x)+\tau(y)$ for all $x, y \in A_+$,
\item $\tau(\lambda x)=\lambda\tau(x)$ for any $x \in A_+$ and any $\lambda \geq 0$, where $0.(+\infty)=0$,
\item $\tau(sxs)=\tau(x)$ for all $x \in A_+$ and all arbitrary symmetry $s$ of $A$.
\end{enumerate}
The trace $\tau$ is said to be faithful if $\tau(x)>0$ for all non-zero $x \in A_+$, finite if $\tau(1) < + \infty$, semifinite if given any $x \in A_+$ there is a non-zero $y \in A_+$, $y \leq x$ with $\tau(y) < +\infty$. The trace $\tau$ is normal if for every net $(x_\alpha)$ monotone increasing to $x$, ($x_\alpha, x \in A_+$) we have $\tau(x_\alpha) \to \tau(x)$. We refer to \cite{AyA85}, \cite{Ayu82}, \cite{Ayu92}, \cite{Kin83} an \cite{PeS82} for more information on traces on $\JBW$-algebras.

Every finite trace on a $\JBW$-algebra $A$ can be extended by the linearity to a linear functional on $A$. Thus a finite trace on a $\JBW$-algebra $A$ is a positive linear functional $\tau$ satisfying the condition $\tau(sxs)=\tau(x)$ for all $x \in A$ and all symmetries $s \in A$. By \cite[Lemma 5.18 p.~147]{AlS03}, it is known that the last condition is equivalent to the formula $\tau (a \circ(b \circ c))=\tau((a \circ b) \circ c)$ for all $a,b,c \in A$. By complexification, we obtain a positive linear functional on the associated $\JBW^*$-algebra $\cal{M}$ satisfying 
\begin{equation}
\label{trace}
\tau (a \circ(b \circ c))
=\tau((a \circ b) \circ c), \quad a,b,c \in \cal{M}.
\end{equation}

It is known that the existence of a normal semifinite faithful trace on a $\JW$-algebra $A$ is characterized by a condition on $A$ called local modularity, see \cite[Theorem 1.2.6 p.~33]{ARU97}. The following is \cite[Corollary 1.2.10 p.~35]{ARU97}.

\begin{prop}
\label{prop-Ayu1210}
Let $A$ be a reversible $\JW$-algebra with a normal trace $\tau$. Then the trace $\tau$ can be extended to a normal trace $\tau_1$ on the von Neumann algebra $A''$. If $\tau$ is faitbful (respectively finite or semifinite) then $\tau_1$ is also faithful (respectively finite or semifinite).
\end{prop}








\begin{example} \normalfont
\label{ex-trace-spin} Consider a spin factor $S=\R\Id \oplus \ovl{\cal{P}}$ as in Example \ref{spin-factor}. By \cite[Lemma 5.21 p.~149]{AlS03} (see also \cite[Proposition 6.1.7 p.~137]{HOS84} and \cite{Top66}), there exists a unique tracial state $\tau$. Moreover, the same reference shows that $\tau$ is normal, faithful and defined by $\tau(\Id)=1$ and $\tau(x)=0$ for any $x \in \ovl{\cal{P}}$. 
\end{example}

\begin{example} \normalfont
The restriction of the trace of the matrix algebra $\M_3(\O)$ is a faithful normal finite trace on the $\JW$-factor $\mathrm{H}_3(\O)$ by the proof of \cite[Proposition 2.9.2 p.~69]{HOS84}. 
\end{example}

\paragraph{Projections on Jordan algebras} The notion of positivity in a $\JB^*$-algebra $A$ is defined in \cite[p.~9]{CGRP18}. We say that a positive map $T \co A \to A$ on a $\JB^*$-algebra $A$ is faithful if $T(x)=0$ for some $x \in A_+$ implies $x=0$. Similarly to the case of $\mathrm{C}^*$-algebras \cite[p.~116]{Str81}, we introduce the following definition.

\begin{defi}
Let $\cal{N}$ be a (unital) $\JW^*$-subalgebra of a $\JW^*$-algebra $\cal{M}$. A linear map $Q \co \cal{M} \to \cal{M}$ is called a Jordan conditional expectation on $N$ if it is a unital positive map of range $\cal{N}$ which is $\cal{N}$-modular, that is 
\begin{equation}
\label{Def-cond-exp-JBstar}
Q(x \circ Q(y))
=Q(x) \circ Q(y), \quad x,y \in \cal{M}.
\end{equation}
\end{defi}
With $x=1$ and $y \in \cal{N}$, we obtain $Q(y)=Q(1 \circ y) \ov{\eqref{Def-cond-exp-JBstar}}{=} Q(1) \circ y=1 \circ y=y$. It follows that $Q$ is the identity on $\cal{N}$. Consequently, $Q$ is an idempotent mapping, that is a projection.

We introduce a similar definition for maps acting on $\JW$-algebras.

\begin{defi}
\label{Def-Jordan-cond-exp-JW}
Let $B$ be a (unital) $\JW$-subalgebra of a $\JW$-algebra $A$. A map $\cal{E} \co A \to A$ is called a Jordan conditional expectation on $B$ if it is a unital positive map of range $B$ which is $B$-modular, that is 
\begin{equation}
\label{Def-cond-exp-JB}
\cal{E}(x \circ \cal{E}(y))
=\cal{E}(x) \circ \cal{E}(y), \quad x,y \in A.
\end{equation}
\end{defi}

We have the following elementary link between these two notions.

\begin{lemma}
\label{Lemma-complexification-conditional-expectation}
Let $B$ be a (unital) $\JW$-subalgebra of a $\JW$-algebra $A$. Let $\cal{E} \co A \to A$ be a Jordan conditional expectation then $\cal{E}_\C \co A_\C \to A_\C$, $x_1+\i x_2\mapsto \cal{E}(x_1)+\i \cal{E}(x_2)$ is a Jordan conditional expectation on the $\JW^*$-subalgebra $B_\C$.
\end{lemma}

\begin{proof}
Note that $\cal{E}_\C(A_\C)=\cal{E}_\C(A+\i A)=\cal{E}(A)+\i \cal{E}(A)$. Since $\mathcal{E}(A)$ is $\JW$-subalgebra of $A$, we conclude that $\cal{E}_\C(A_\C)$ is a $\JW^*$-subalgebra of the $\JW^*$-algebra of $A_\C$. Finally, if $x=x_1+\i x_2$ and $y=y_1+\i y_2$ belongs to $A_\C$, a simple computation\footnote{\thefootnote. For any $x,y \in A_\C$, we have with obvious notations
\begin{align*}
\MoveEqLeft
\cal{E}_\C(x \circ \cal{E}_\C(y))
=\cal{E}_\C\big((x_1+\i x_2) \circ \cal{E}_\C(y_1+\i y_2)\big)
=\cal{E}(x_1 \circ \cal{E}(y_1)+\i x_1 \circ \cal{E}(y_2) +\i x_2 \circ \cal{E}(y_1) - x_2 \circ \cal{E}(y_2))\\
&=\cal{E}(x_1 \circ \cal{E}(y_1))+\i \cal{E}(x_1 \circ \cal{E}(y_2)) +\i \cal{E}(x_2 \circ \cal{E}(y_1)) - \cal{E}(x_2 \circ \cal{E}(y_2)) \\
&\ov{\eqref{Def-cond-exp-JB}}{=} \cal{E}(x_1) \circ \cal{E}(y_1)+\i \cal{E}(x_1) \circ \cal{E}(y_2) +\i \cal{E}(x_2) \circ \cal{E}(y_1) - \cal{E}(x_2) \circ \cal{E}(y_2) \\
&=\cal{E}_\C(x_1+\i x_2) \circ \cal{E}_\C(y_1+\i y_2)
=\cal{E}_\C(x) \circ \cal{E}_\C(y).
\end{align*}}  gives \eqref{Def-cond-exp-JBstar}.
\end{proof}

The following is a simple consequence of \cite[Proposition 2.2.9]{Sto13} (see also (\cite[Corollary 1.5]{EfS79} 
and \cite{BlN21} for related results) but is fundamental for us. Recall that the definite set $D$ \cite[Definition 2.1.4]{Sto13} of a positive map $Q \co A \to A$ on a $\mathrm{C}^*$-algebra $A$ is defined by 
\begin{equation}
\label{Def-definite-set}
D
\ov{\mathrm{def}}{=} \big\{x \in \cal{M}_\sa : Q(x^2)=Q(x)^2\big\}.
\end{equation}

\begin{prop}
\label{Prop-fundamental}
Let $\cal{M}$ be a von Neumann algebra. Let $Q \co \cal{M} \to \cal{M}$ be a weak* continuous faithful unital positive projection. Then $Q$ is a Jordan conditional expectation and the range $Q(\cal{M})$ is a $\JW^*$-subalgebra of $(\cal{M},\circ)$.
\end{prop}

\begin{proof}
Note that $(\cal{M}_\sa,\circ)$ is a $\JW$-algebra. So the restriction $Q|\cal{M}_\sa \co \cal{M}_\sa \to \cal{M}_\sa$ is a weak* continuous faithful unital positive projection. By \cite[Proposition 2.2.8]{Sto13} applied with $e=1$, $Q(\cal{M}_\sa)$ is a $\JW$-subalgebra\footnote{\thefootnote. Note that \cite[Theorem 2.2.2]{Sto68} says that $Q(M_\sa)$ is a $\JC$-subalgebra of $M$.} of $(M_\sa,\circ)$. Moreover, by \cite[Theorem 2.2.2]{Sto13}, $Q(\cal{M}_\sa)$ is contained in the definite set of $Q$. Using \cite[Proposition 2.1.7 (i)]{Sto13} in the first equality, we infer that
\begin{equation*}
Q(x \circ Q(y))
=Q(x) \circ Q^2(y)
=Q(x) \circ Q(y), \quad x,y \in \cal{M}_\sa.
\end{equation*}
Hence $Q|\cal{M}_\sa \co \cal{M}_\sa \to \cal{M}_\sa$ is a Jordan conditional expectation. With  Lemma \ref{Lemma-complexification-conditional-expectation}, we conclude that $Q=(Q|\cal{M}_\sa)_\C \co \cal{M} \to \cal{M}$ is a Jordan conditional expectation on $(Q(\cal{M}_\sa))_\C=Q(\cal{M})$.
\end{proof}





\paragraph{Selfadjoint maps} Let $\cal{M}$ be a von Neumann algebra equipped with a normal semifinite faithful trace $\tau$. Recall that a \textit{positive} normal contraction $T \co \cal{M} \to \cal{M}$ is selfadjoint with respect to $\tau$ \cite[p.~49]{JMX06} if for any $x,y \in \cal{M} \cap \L^1(\cal{M})$ we have $\tau(T(x)y)=\tau(xT(y))$. We have a similar notion for a normal state $\varphi$ on $\cal{M}$ instead of the trace $\tau$, see \cite[p.~122]{JMX06}. It is important to note that if $\varphi$ is a normal state there is nothing to ensure that a $\varphi$-preserving normal Jordan conditional expectation $Q \co \cal{M} \to \cal{M}$ is selfadjoint contrary to the case of classical conditional expectations on von Neumann algebras. But we will show that such a map is Jordan-selfadjoint :
$$
\varphi(Q(x) \circ y)=\varphi(x \circ Q(y)), \quad x,y \in \cal{M}.
$$ 
It is a difficulty for identify the range of its $\L^p$-extension in Section \ref{Sec-a-precise}. Nevertheless, we have the following fundamental observation which will be used in Section \ref{Sec-a-precise}.

\begin{prop}
\label{Prop-selfadjoint}
Let $\cal{M}$ be a von Neumann algebra equipped with a normal semifinite faithful trace. Let $Q \co \cal{M} \to \cal{M}$ be a trace preserving normal Jordan conditional expectation. Then $Q$ is selfadjoint.   
\end{prop}

\begin{proof} 
Using the preservation of the trace by $Q$ in the third and the sixth equalities, for any $x,y \in \cal{M} \cap \L^1(\cal{M})$, we obtain
\begin{align*}
\MoveEqLeft
\tau(Q(x)\circ y)
=\tau(Q(Q(x)\circ y)) 
\ov{\eqref{Def-cond-exp-JBstar}}{=} \tau(Q(x)\circ Q(y)) 
\ov{\eqref{Def-cond-exp-JBstar}}{=} \tau(Q(x \circ Q(y)))
=\tau(x \circ Q(y)).            
\end{align*}
\end{proof}

Let $A$ be a $\JBW$-algebra equipped with a normal faithful state $\varphi$. We say that a normal map $Q \co A \to A$ is Jordan-selfadjoint if 
$$
\varphi(T(x) \circ y)
=\varphi(x \circ T(y)), \quad x,y \in A. 
$$

The following elementary observation is elementary but is crucial for us. This is a consequence of the computation $\tau(x \circ y)=\frac{1}{2}\tau(xy+yx)=\tau(xy)$ which is true for any $x,y \in \cal{M}$. 

\begin{prop}
\label{Prop-Jordan-se=se}
Let $\cal{M}$ be a von Neumann algebra equipped with a normal finite faithful trace $\tau$. A normal map $T \co \cal{M} \to \cal{M}$ is Jordan-selfadjoint map\footnote{\thefootnote. Here $\cal{M}$ is equipped with its canonical structure of $\JW$-algebra.} if and only if $T$ is selfadjoint.
\end{prop}

 %

\paragraph{Existence of Jordan conditional expectations}
Let $A$ be a $\JBW$-algebra and $B$ a $\JBW$-subalgebra of $A$. Suppose that $\tau$ is a faithful normal tracial state on $A$. If we also denote by $\tau$ the restriction of $\tau$ on $B$, it is essentially showed in \cite[Theorem 4.2]{HaS95} (combined with \cite[Remark 3.7]{HaH84}) that there exists a faithful normal Jordan conditional expectation $Q \co A \to A$ onto $B$ such that $\tau \circ Q=\tau$. See \cite[Theorem p.~78]{Edw86} for a previous prelimiary result without proof. By complexification, we obtain the following result.

\begin{prop}
\label{prop-Jordan-conditional-expectation-existence-bis}
Let $\cal{M}$ and $\cal{N}$ be $\JBW^*$-algebras such that $\cal{N}$ is a subalgebra of $\cal{M}$. Let $\tau$ be a normalized normal finite faithful trace on $\cal{M}$. Then there exists a trace preserving normal faithful Jordan conditional expectation $Q \co \cal{M} \to \cal{M}$ on $\cal{N}$.
\end{prop}



\section{A lifting of contractive positive projections on noncom. $\L^p$-spaces}
\label{Sec-a-local}

In this section, we prove Theorem \ref{main-th-ds-intro-bis}. We follow \cite{ArR19} with some clarifications. The dependence with respect to $\varphi$ of the Haagerup trace $\tr_\varphi$ defined in \eqref{Def-tr} is the source of technical complications. 

Suppose $1 < p < \infty$. Let $\cal{M}$ be a $\sigma$-finite ($=$ countably decomposable) von Neumann algebra and $P \co \L^p(\cal{M}) \to \L^p(\cal{M})$ be a positive contractive projection. We define the support $s_P$ of $\Ran P$ as the supremum in $\cal{M}$ of the supports of the positive elements in $\Ran P$:
\begin{equation}
\label{support-s(P)}
s_P\ov{\mathrm{def}}{=}\bigvee_{h \in \Ran P, h \geq 0} s(h).
\end{equation}

\begin{lemma}
\label{Lemma-sP}
For any $y \in \L^p(\cal{M})$ we have $P(y)=P(s_Pys_P)$.
\end{lemma}

\begin{proof}
Note that $P^* \co \L^{p^*}(\cal{M}) \to \L^{p^*}(\cal{M})$ is also a positive contractive projection. Recall again that a noncommutative $\L^p$-space is a smooth strictly convex reflexive Banach space (if $1<p<\infty$). By Lemma \ref{Lemma-smooth}, the map $J_{\L^p(\cal{M})}$ of \eqref{comput-J} induces a bijection from $\Ran P$ onto $\Ran P^*$ and this map and its inverse preserve the positivity. This remark and the formula \eqref{support-s(P)} imply that  $s_{P^*}=s_P$. Thus for every $y \in \L^p(\cal{M})$, we obtain that
\begin{align*}
\MoveEqLeft
\bnorm{P(y)}_{\L^p(M)}^2
\ov{\eqref{Def-JX}}{=} \tr\big(P(y)J_{\L^p(M)}(P(y))\big)
=\tr\big(yP^*(J_{\L^p(M)}(P(y)))\big)
=\tr\big(yJ_{\L^p(M)}(P(y))\big) \\ 
&=\tr\big(y s_{P^*}J_{\L^p(M)}(P(y))s_{P^*}\big)
=\tr\big(y s_PJ_{\L^p(M)}(P(y))s_P\big)
= \tr\big(s_P y s_P J_{\L^p(M)}(P(y))\big).
\end{align*}
Hence if $y s_P=0$ or $s_Py=0$ we have $P(y)=0$. Now, for any $y \in \L^p(\cal{M})$, we obtain\footnote{\thefootnote. Note that $s_P(1-s_P)y(1-s_P)=0$ and similarly for the others.} 
\begin{align*}
P(y)
&=P\big((1-s_P\big)y\big(1-s_P)\big)+ P\big((1-s_P)ys_P\big)+ P\big(s_P y(1-s_P)\big)+P(s_Pys_P) \\
&=P(s_Pys_P).
\end{align*}
\end{proof}

Let $\chi$ be a faithful normal state on $\cal{M}$ which exists by \cite[Exercise 7.6.46]{KaR97b}. Here, we use the concrete realization $\L^p(\cal{M})=\L^p(\cal{M},\chi)$. The following is \cite[Proposition 3.4]{ArR19}

\begin{lemma}
\label{lemma-second}
There exists a positive element $k$ of $\Ran P$ such that $s(k)=s(P)$. 
\end{lemma}

We consider a positive element $k$ as in Lemma \ref{lemma-second}. We have $P(k)=k$ and $k \in \L^p(\cal{M},\chi)_+$. We can suppose that $k \not=0$. Note that $s(k)Ms(k) \to \C$, $x \mapsto \tr_{\chi}(k^p x)$ is a faithful normal linear functional on $s(k)Ms(k)$. Using the procedure \eqref{Extension-weight1}, we can consider a normal faithful linear functional $\varphi$ on $\cal{M}$ such that $s(k)$ belongs to the centralizer of $\varphi$ and such that the reduced state $\varphi_{s(k)}$ on $s(k)Ms(k)$ satisfies
\begin{equation}
\label{Divers-2367}
\varphi_{s(k)}(x)
=\tr_{\chi}(k^p x), \quad x \in s(k)Ms(k).
\end{equation}
Multiplying $k$ by a constant, we can suppose that $\varphi$ is a state. From \eqref{kappa}, we have a canonical map $\kappa$ which induces an order and isometric identification $\kappa \co \L^p(\cal{M},\chi) \to \L^p(\cal{M},\varphi)$ for all $p$. We let
\begin{equation}
\label{Def-h-psi-h}
h \ov{\mathrm{def}}{=} \kappa(k),
\qquad \psi \ov{\mathrm{def}}{=}\varphi_{s(k)}
\end{equation}
and $M_h \ov{\mathrm{def}}{=}s(h)\cal{M}s(h)$. By transport of structure, we have\footnote{\thefootnote. Here $P \co \L^p(\cal{M},\varphi) \to \L^p(M,\varphi)$.} 
\begin{equation}
\label{Equa-inter-6789-1}
P(h)=h 
\quad \text{and} \quad s(h)=s(k).
\end{equation}
In particular, with Lemma \ref{Lemma-sP} and Lemma \ref{lemma-second}, we obtain the first point of Theorem \ref{main-th-ds-intro-bis}. Furthermore, For any $x \in \cal{M}_h$ we have
\begin{equation}
\label{Equa-inter-6789}
\psi(x)
\ov{\eqref{Def-h-psi-h}}{=} \varphi_{s(k)}(x)
\ov{\eqref{Divers-2367}}{=} \tr_{\chi}(k^p x)
\ov{\eqref{Trace-preserving}}{=} \tr_{\varphi}(\kappa(k^p x))
=\tr_{\varphi}(\kappa(k)^p x)
\ov{\eqref{Def-h-psi-h}}{=} \tr_{\varphi}(h^p x).
\end{equation}
Since $s(h)$ belongs to the centralizer of $\varphi$, the noncommutative $\L^p$-space $\L^p(\cal{M}_h) \ov{\mathrm{def}}{=} \L^p(\cal{M}_h,\psi)$ can be identified order and isometrically with the subspace $s(h)\L^p(M,\varphi)s(h)$ of $\L^p(\cal{M},\varphi)$. By applying Theorem \ref{Th-relevement-cp} to the restriction $P|_{\L^p(\cal{M}_h)} \co \L^p(\cal{M}_h) \to \L^p(\cal{M})$ and to the positive element $h$ of $\L^p(\cal{M}_h)$ which has support $s(h)=1_{\cal{M}_h}$, we see that there exists a unique linear map $Q \co \cal{M}_h \to s(P(h))\cal{M}s(P(h))=\cal{M}_h$ such that
\begin{equation}
\label{relevement-E_h}
\quad P\big(h^{\frac{1}{2}}xh^{\frac{1}{2}}\big)
=h^{\frac{1}{2}}Q(x)h^{\frac{1}{2}}, \qquad x \in \cal{M}_h.
\end{equation}
Moreover, this map $Q$ is unital, contractive, normal and positive. 

\begin{lemma}
\label{Lemma-faithful}
The map $Q \co \cal{M}_h \to \cal{M}_h$ is faithful.
\end{lemma}

\begin{proof}
Recall that for $1<p<\infty$, the norm of the space $\L^p(\cal{M})$ 
is strictly monotone\footnote{\thefootnote. That means that if $0 \leq x \leq y$ with $x \not=y$ then we have $\norm{x}_p < \norm{y}_p$.}. Now, we will show that if $h_0 \in \L^p(\cal{M})$ satisfy $0 \leq h_0 \leq h$ and $P(h_0)=0$ then $h_0=0$. We have
$$
h
\ov{\eqref{Equa-inter-6789-1}}{=} P(h)
=P(h)-P(h_0)
=P(h-h_0)
$$ 
Since $P$ is contractive, we deduce that $\norm{h}_p = \norm{P(h-h_0)}_p \leq \norm{h-h_0}_p$. Since $0 \leq h-h_0 \leq h$ we infer that $\norm{h}_p=\norm{h-h_0}_p$ and finally $h_0=0$ by strict monotonicity of the $\L^p$-norm. 

Now, we can deduce that $Q$ is faithful. Indeed, if $x \in \cal{M}_h^+$ and $Q(x)=0$ we have
$$
P\big(h^{\frac{1}{2}}xh^{\frac{1}{2}}\big)
\ov{\eqref{relevement-E_h}}{=} h^{\frac{1}{2}} Q(x) h^{\frac{1}{2}}
=0.
$$
By \cite[1.6.9]{Dix77}, we have $0 \leq x \leq \norm{x}_\infty$ so $0 \leq h^{\frac{1}{2}}xh^{\frac{1}{2}} \leq \norm{x}_{\infty}h$. We see that $h^{\frac{1}{2}}xh^{\frac{1}{2}}=0$ by the first part of the proof. Since $\cal{M}_h=s(h)Ms(h)$, we conclude by Lemma \ref{lemma2-GL} that $x=0$. 
\end{proof}

\begin{lemma}
The map $Q \co \cal{M}_h \to \cal{M}_h$ is a projection.
\end{lemma}

\begin{proof}
For any $x \in \cal{M}_h$, we have 
$$
P\big(h^{\frac{1}{2}}xh^{\frac{1}{2}}\big)
=P^2\big(h^{\frac{1}{2}}xh^{\frac{1}{2}}\big)
\ov{\eqref{relevement-E_h}}{=} P\big(h^{\frac{1}{2}} Q(x) h^{\frac{1}{2}}\big)
\ov{\eqref{relevement-E_h}}{=} h^{\frac{1}{2}} Q^2(x) h^{\frac{1}{2}}.
$$
Using the uniqueness of $Q$ given by Theorem \ref{Th-relevement-cp}, we infer that $Q^2=Q$, i.e. $Q$ is a projection. 
\end{proof}

Now, we prove that $Q$ is $\psi$-invariant i.e. the third point of Theorem \ref{main-th-ds-intro-bis}. 

\begin{lemma} 
\label{Lemma-prservation-state}
We have $\psi \circ Q=\psi$.
\end{lemma}  

\begin{proof}
Since $h$ is positive, by Lemma \ref{Lemma-J-lp}, we have $J_{\L^p(\cal{M})}(h)=\norm{h}^{2-p}_{p}h^{p-1}$. By \cite[Corollary 5.2]{PiX03}, the Banach space $\L^p(\cal{M}_h)$ is smooth and strictly convex. Using the contractive dual map $P^* \co \L^{p^*}(\cal{M}) \to \L^{p^*}(\cal{M})$ and Lemma \ref{Lemma-smooth}, we see that $P^*(\norm{h}^{2-p}_{p}h^{p-1})=\norm{h}^{2-p}_{p}h^{p-1}$, that is
\begin{equation}
\label{equa-inter-400}
P^*(h^{p-1})
=h^{p-1}.
\end{equation}
For any $k \in \L^p(\cal{M}_h)$, it follows that
\begin{equation}
\label{equa-inter-401}
\tr_\varphi\big(h^{p-1}P(k)\big)
=\tr_\varphi\big(P^*(h^{p-1})k\big)
\ov{\eqref{equa-inter-400}}{=}\tr_\varphi(h^{p-1} k).
\end{equation}
In particular, for any $x \in \cal{M}_h$, we have
\begin{align*}
\MoveEqLeft
\psi(Q(x)) \ov{\eqref{Equa-inter-6789}}{=} \tr_\varphi\big(h^pQ(x)\big) 
=\tr_\varphi\big(h^{p-1}(h^{\frac{1}{2}}Q(x) h^{\frac{1}{2}})\big)
\ov{\eqref{relevement-E_h}}{=} \tr_\varphi\big(h^{p-1}P(h^{\frac{1}{2}} x h^{\frac{1}{2}})\big) \\
&\ov{\eqref{equa-inter-401}}{=} \tr_\varphi\big(h^{p-1}(h^{\frac{1}{2}} x h^{\frac{1}{2}})\big)
=\tr_\varphi(h^p x)
\ov{\eqref{Equa-inter-6789}}{=} \psi(x).
\end{align*}
\end{proof}

So, we have proved the ``only if'' part of Theorem \ref{main-th-ds-intro-bis}.

\vspace{0.2cm}

Conversely, suppose that the conditions of Theorem \ref{main-th-ds-intro-bis} are satisfied. We can suppose $\norm{h}_p=1$. We introduce the reduced weight $\psi$ on the von Neumann algebra $\cal{M}_h \ov{\mathrm{def}}{=} s(h)\cal{M}s(h)$ induced by the state $\tr_{\varphi}(h^p\,\cdot)$.

In the sequel, we will use the density operator $h_{\psi} \in \L^1(\cal{M}_h,\psi)$ associated with the weight $\psi$ on $\cal{M}_h$. From \eqref{kappa}, we have a canonical map $\kappa$ which induces an order and isometric identification $\kappa \co \L^p(\cal{M}_h,\psi) \to \L^p(\cal{M}_h,\varphi_{s(h)})$ for all $p$. If $x \in \cal{M}_h$, using Lemma \ref{Lemme-trace coincide} in the third equality, we see that
\begin{equation}
\label{Magic-Formula}
\tr_{\psi}(h_{\psi} x) 
\ov{\eqref{HJX-1.13}}{=} \psi(x) 
= \tr_{\varphi}(h^p x)
=\tr_{\varphi_{s(h)}}(h^p x)
\ov{\eqref{Trace-preserving}}{=} \tr_{\psi}(\kappa^{-1}(h^p x))
=\tr_{\psi}(\kappa^{-1}(h)^p x).
\end{equation}
We conclude that
\begin{equation}
\label{Magic-formula-2}
h_{\psi}
=\kappa^{-1}(h)^p.
\end{equation}
With the condition 3 of Theorem \ref{main-th-ds-intro-bis}, we can consider by \eqref{Map-extension-Lp} with $\cal{M}$ instead of $\mathfrak{m}_\varphi$ the contractive positive operator $Q_p \co \L^p(\cal{M}_h,\psi) \to \L^p(\cal{M}_h,\psi)$ induced by the map $Q \co \cal{M}_h \to \cal{M}_h$ and defined by
\begin{equation}
\label{Def-esperance-h}
Q_p\big(h_{\psi}^{\frac{1}{2p}} xh_{\psi}^{\frac{1}{2p}}\big)
\ov{\eqref{Map-extension-Lp}}{=} h_{\psi}^{\frac{1}{2p}}Q(x)h_{\psi}^{\frac{1}{2p}},\quad x \in \cal{M}_h. 
\end{equation}
For any $x \in \cal{M}_h$, note that
\begin{equation*}
\label{}
Q_p^2\big(h_{\psi}^{\frac{1}{2p}} xh_{\psi}^{\frac{1}{2p}}\big)
\ov{\eqref{Def-esperance-h}}{=} Q_p\big(h_{\psi}^{\frac{1}{2p}} Q(x)h_{\psi}^{\frac{1}{2p}}\big)
\ov{\eqref{Def-esperance-h}}{=}h_{\psi}^{\frac{1}{2p}} Q^2(x)h_{\psi}^{\frac{1}{2p}}
=h_{\psi}^{\frac{1}{2p}} Q(x)h_{\psi}^{\frac{1}{2p}}
\ov{\eqref{Def-esperance-h}}{=} Q_p\big(h_{\psi}^{\frac{1}{2p}} xh_{\psi}^{\frac{1}{2p}}\big).
\end{equation*}
We deduce that $Q_p^2=Q_p$, i.e. that the map $Q_p$ is a projection. Since $s(h)$ \textit{belongs} to the centralizer of $\varphi$, we have a order isometric identification of $\L^p\big(\cal{M}_h,\varphi_{s(h)}\big)$ in the space $\L^p(\cal{M},\varphi)$.

For any $x \in \cal{M}_h$, we have
\begin{align*}
\MoveEqLeft
P\kappa\big(h_{\psi}^{\frac{1}{2p}}xh_{\psi}^{\frac{1}{2p}}\big) 
\ov{\eqref{Magic-formula-2}}{=} P\kappa\big(\kappa^{-1}(h)^{\frac{1}{2}}x \kappa^{-1}(h)^{\frac{1}{2}}\big)
= P\big(h^{\frac{1}{2}}x h^{\frac{1}{2}}\big)
\ov{\eqref{relevement-E_h-intro}}{=} h^{\frac{1}{2}}Q(x) h^{\frac{1}{2}} \\
&\ov{\eqref{Magic-formula-2}}{=} \kappa(h_{\psi})^{\frac{1}{2p}}Q(x) \kappa(h_{\psi})^{\frac{1}{2p}} 
=\kappa\big(h_{\psi}^{\frac{1}{2p}}Q(x) h_{\psi}^{\frac{1}{2p}}\big) \ov{\eqref{Def-esperance-h}}{=} \kappa Q_p\big(h_{\psi}^{\frac{1}{2p}}x h_{\psi}^{\frac{1}{2p}}\big).
\end{align*}  
Hence, by density we conclude that we have the following commutative diagram.
$$
\xymatrix @R=1cm @C=2cm{
\L^p(\cal{M},\varphi)\ar@{^{(}->}[r]^{P} &\L^p(\cal{M},\varphi)\\
\L^p\big(\cal{M}_h,\varphi_{s(h)}\big)\ar@{^{(}->}[u]  &  \L^p\big(\cal{M}_h,\varphi_{s(h)}\big)\ar@{^{(}->}[u]		\\
\L^p(\cal{M}_h,\psi)  \ar[u]^{\kappa}   \ar[r]^{Q_p} & \L^p(\cal{M}_h,\psi) \ar[u]_{\kappa} \\			
}
$$
In particular, we have the inclusion $P(\L^p(\cal{M}_h,\varphi_{s(h)})) \subset \L^p(\cal{M}_h,\varphi_{s(h)})$ and in addition the restriction $P|_{\L^p(\cal{M}_h)} \co \L^p(\cal{M}_h) \to \L^p(\cal{M}_h)$ is a positive contractive projection where we use the notation $\L^p(\cal{M}_h)=\L^p(\cal{M}_h,\varphi_{s(h)})$. Now, we consider the positive contractive map $R \co \L^p(\cal{M},\varphi) \to s(h)\L^p(\cal{M})s(h)$, $z \mapsto s(h)zs(h)$ and the canonical isometry $j \co \L^p(\cal{M}_h) \to \L^p(\cal{M})$. Note that $R \circ j=\Id_{\L^p(\cal{M}_h)}$. Then the point 1 of Theorem \ref{main-th-ds-intro-bis} says that 
\begin{equation}
\label{Facto-de-P}
P
=j \circ P|_{\L^p(\cal{M}_h)} \circ R.
\end{equation}
Now, we deduce that
\begin{align*}
\MoveEqLeft
P^2            
\ov{\eqref{Facto-de-P}}{=} j P|_{\L^p(\cal{M}_h)} R j  P|_{\L^p(\cal{M}_h)}  R
=j \big(P|_{\L^p(\cal{M}_h)}\big)^2 R
=j P|_{\L^p(\cal{M}_h)}R
\ov{\eqref{Facto-de-P}}{=} P.
\end{align*} 
We conclude that $P$ is a projection. The formula \eqref{Facto-de-P} shows that $P$ is positive and contractive. 

The first part of the last sentence of Theorem \ref{main-th-ds-intro-bis} is a consequence of \eqref{relevement-E_h} with $x=1$ and the second part can be deduced from Proposition \ref{Prop-fundamental} which also says that $Q$ is a Jordan conditional expectation. The proof is complete.

\begin{remark} \normalfont
\label{Rem1}
If the contractive projection $P \co \L^p(\cal{M}) \to \L^p(\cal{M})$ is completely copositive, then by Theorem \ref{Th-relevement-cp}, the previous faithful positive normal unital projection $Q_h \co \cal{M}_h \to \cal{M}_h$ is also completely copositive, hence decomposable within the meaning\footnote{\thefootnote. A map is decomposable if it is the sum of a completely positive map and a completely copositive map. This notion is different of the one of \cite {ArK23} and \cite{Arh22}.} of \cite[Definition 1.2.8]{Sto13}. By \cite[Theorem 2.2.4]{Sto13}, we conclude that the $\JW$-algebra $Q((\cal{M}_h)_\sa)$ is necessarily reversible. 

\end{remark}

\begin{remark} \normalfont
\label{Rem2}
The paper \cite{CNR04} (see also \cite{BuP02} for related facts) furnish additional information on the range $Q(s(h)\cal{M}s(h))$ if we have additional knowledge on the von Neumann algebra $\cal{M}$. More precisely, we have the following properties.
\begin{enumerate}
\item[i.] If $\cal{M}$ is of type $\I$ then by combining \cite[Proposition 6.7.2]{Li92} and \cite[Proposition 2.8]{CNR04}, we see that the $\JW$-algebra $Q((s(h)\cal{M}s(h))_\sa)$ is also of type $\I$.

\item[ii.] If $\cal{M}$ is semifinite then with \cite[Proposition 6.5.9]{Li92} and \cite[Proposition 2.7]{CNR04} we infer that the $\JW$-algebra $Q((s(h)\cal{M}s(h))_\sa)$ is also semifinite.

\item[iii.] If $\cal{M}$ is finite then using \cite[Proposition 6.3.1]{Li92} and \cite[Proposition 2.7]{CNR04} we deduce that the $\JW$-algebra $Q((s(h)\cal{M}s(h))_\sa)$ is finite, i.e. modular.
\end{enumerate}	
\end{remark}

\begin{remark} \normalfont
\label{Rem3} The case where the von Neumann algebra $\cal{M}$ is finite and equipped with a normalized normal finite faifthful trace and where the projection $P \co \L^p(\cal{M}) \to \L^p(\cal{M})$ satisfies $P(1)=1$ is much simpler and instructive. Indeed, the previous proof shows that the map $P$ is equal to the $\L^p$-extension $Q_p \co \L^p(\cal{M}) \to \L^p(\cal{M})$ of a trace preserving normal faithful Jordan conditional expectation $Q \co \cal{M}\to \cal{M}$ which is reminiscent of the classical result \cite[Corollary 5.53 p.~222]{AbA02} which says that a positive contractive projection $P \co \L^p(\Omega) \to \L^p(\Omega)$ on a classical $\L^p$-space associated to a probability space $\Omega$ which makes constant invariant is induced by a conditional expectation.
\end{remark}

\begin{remark} \normalfont
Our assumption of $\sigma$-finiteness can be removed using weights and a gluing argument as in \cite{ArR19}.
\end{remark}


\begin{remark} \normalfont
If $\cal{M}$ is commutative, we recover essentially the structure of contractive projections of \cite[Theorem 1]{Lac74} in the positive case.
\end{remark}

Note the following module map property twisted by the modular group $(\sigma_{t}^\varphi)_{t \in \R}$ which shows a new phenomenon. Recall that an element $x$ of $\cal{M}$ is an entire analytic vector with respect to the normal faithful state $\varphi$ if the function $t \mapsto \sigma_t^\varphi(x)$ extends to a entire function from $\mathbb{C}$ into $\cal{M}$. By \cite[p.~32]{Str81}, the family $\cal{M}_a$ of entire analytic vectors is a weak* dense $*$-subalgebra of the von Neumann algebra $\cal{M}$. 
A folklore result used in the proofs of \cite[Lemma 1.1]{JuX03} and \cite[Proposition 5.5]{HJX10} relying on \cite[Proposition 1.4]{GoL99} says that if $x \in \cal{M}_a$ we have 
\begin{equation}
\label{first-commutation}
xh_\varphi^\alpha
=h_\varphi^\alpha\sigma_{\i\alpha}^\varphi(x)
\end{equation}
for any $\alpha \in \mathbb{C}$ with $\Re \alpha >0$.


\begin{prop}
\label{Prop-modular}
Let $\cal{M}$ be a von Neumann algebra equipped with a normal faithful state $\varphi$. Suppose $1 \leq p \leq \infty$. Let $Q \co \cal{M} \to \cal{M}$ be a normal Jordan conditional expectation onto a $\JW^*$-subalgebra $Q(\cal{M})$ of the $\JW^*$-algebra $(\cal{M},\circ)$ such that $\varphi=\varphi \circ Q$. Then for any entire analytic vector $a$ belonging to $Q(\cal{M})$ such that $\sigma_{\frac{\i}{2p}}^\varphi(a)$ also belongs to $Q(\cal{M})$ (any element of $Q(\cal{M})$ if $p=\infty$) we have
\begin{equation}
\label{}
Q_p\big(ax\sigma_{\frac{\i}{p}}^\varphi(a)\big)
=aQ_p(x)\sigma_{\frac{\i}{p}}^\varphi(a), \quad x \in \L^p(\cal{M}).
\end{equation}
\end{prop}

\begin{proof}
The case $p=\infty$ is a consequence of the formula $aba=2(a \circ b) \circ a-a^2 \circ b$. Suppose $1 \leq p <\infty$. For any $y \in \cal{M}$, using the case $p=\infty$ in the fourth equality, we have
 \begin{align*}
\MoveEqLeft
Q_p\big(ah_\varphi^{\frac{1}{2p}} y h_\varphi^{\frac{1}{2p}}\sigma_{\frac{\i}{p}}^\varphi(a)\big)          
\ov{\eqref{first-commutation}}{=} Q_p\big(h_\varphi^{\frac{1}{2p}} \sigma_{\frac{\i}{2p}}^\varphi(a)y\sigma_{-\frac{\i}{2p}}^\varphi(\sigma_{\frac{\i}{p}}^\varphi(a)) h_\varphi^{\frac{1}{2p}}\big) \\
&=Q_p\big(h_\varphi^{\frac{1}{2p}} \sigma_{\frac{\i}{2p}}^\varphi(a)y\sigma_{\frac{\i}{2p}}^\varphi(a) h_\varphi^{\frac{1}{2p}}\big) 
\ov{\eqref{Map-extension-Lp}}{=} h_\varphi^{\frac{1}{2p}} Q\big(\sigma_{\frac{\i}{2p}}^\varphi(a)y\sigma_{\frac{\i}{2p}}^\varphi(a)\big) h_\varphi^{\frac{1}{2p}} \\
&=h_\varphi^{\frac{1}{2p}} \sigma_{\frac{\i}{2p}}^\varphi(a)Q(y)\sigma_{\frac{\i}{2p}}^\varphi(a) h_\varphi^{\frac{1}{2p}}
\ov{\eqref{first-commutation}}{=} ah_\varphi^{\frac{1}{2p}} Q(y) h_\varphi^{\frac{1}{2p}}\sigma_{\frac{\i}{p}}^\varphi(a)
\ov{\eqref{Map-extension-Lp}}{=} aQ_p(h_\varphi^{\frac{1}{2p}}yh_\varphi^{\frac{1}{2p}})\sigma_{\frac{\i}{p}}^\varphi(a).
\end{align*} 
We conclude by density.
\end{proof}

\begin{remark} \normalfont
The assumption that $\sigma_{\frac{\i}{2p}}^\varphi(a)$ also belongs to $Q(\cal{M})$ is maybe useless. Note that the space of entire analytic vectors of $Q(\cal{M})$ is weak* dense in $Q(\cal{M})$. We sketch the argument. By \cite[Theorem 4.2]{HaS95}, the $\JW$-algebra $A$ associated to $Q(\cal{M})$ is invariant under the one-parameter cosine family $(\rho_t^\varphi)_{t \in \R}$ associated with $\varphi$. Since $\cal{M}$ is the von Neumann algebra, recall that by \cite[Proposition 3.6]{HaH84} the cosine family $(\rho_t^\varphi)_{t \in \R}$ is given by $\rho_t^\varphi(x)=\frac{1}{2}(\sigma^\varphi_t(x)+\sigma^\varphi_{-t}(x))$ where $x \in A$. Now, for any $\gamma >0$ the element
\begin{align*}
\MoveEqLeft
x_\gamma
\ov{\mathrm{def}}{=}\frac{1}{(2\pi\gamma^2)^{1/2}} \int_\R \e^{-\frac{t^2}{2\gamma^2}} \sigma_t^\varphi(x)\d t
=\frac{1}{(2\pi\gamma^2)^{1/2}} \int_0^{\infty} \e^{-\frac{t^2}{2\gamma^2}} (\sigma^\varphi_t(x)+\sigma^\varphi_{-t}(x))\d t
\end{align*}
belongs to $A$ and approximate $x \in A$ in the weak* topology when $\gamma \to 0$ by \cite[p.~72]{Sun87} \cite[p.~33]{Str81}. 
\end{remark}

We have the following (easy) particular case.

\begin{cor}
\label{car-modular-tracial}
Let $\cal{M}$ be a von Neumann algebra equipped with a normal finite faithful trace $\tau$. Suppose $1 \leq p \leq \infty$. Let $Q \co \cal{M} \to \cal{M}$ be a normal Jordan conditional expectation onto a $\JW^*$-subalgebra $\cal{N}$ of the $\JW^*$-algebra $(\cal{M},\circ)$ such that $\tau=\tau \circ Q$. Then  
\begin{equation}
\label{}
Q_p(axa)
=aQ_p(x)a, \quad a \in Q(\cal{M}), x \in \L^p(\cal{M}).
\end{equation}
\end{cor}

\section{An analysis of the lifting contractive projection}
\label{Sec-a-precise}

Let $\cal{M}$ be a von Neumann algebra and $\varphi$ be a normal faithful positive linear form on $\cal{M}$. Let $Q \co \cal{M} \to \cal{M}$ be a $\varphi$-preserving normal unital positive projection on a $\JW^*$-subalgebra $\cal{N}$. We let $A \ov{\mathrm{def}}{=} \cal{N}_\sa$ which is a $\JW$-subalgebra of $(\cal{M}_\sa,\circ)$. In this section, we will make an analysis of the structure of the $\L^p$-extension of $Q$, using results of Haagerup and St\o rmer. We also obtain information on the range.

We begin to use structure results on the $\JW$-algebra $A$. By \cite[Theorem 6.4]{Sto66}, \cite[Lemma 2.3]{Sto68}, there exist projections $e_1,e_2,e_3$ in the center $\Zc(A)$ of the $\JW$-algebra $A$ with sum $1$ such that
\begin{enumerate}
	\item $A_1=e_1A$ is the selfadjoint part of the von Neumann algebra $A_1''$,
	\item $A_2=e_2A$ is purely real\footnote{\thefootnote. In particular $A_2$ is the selfadjoint part of the real von Neumann algebra $\kR(A_2)$.}, 
	\item $A_3=e_3A$ is totally non reversible, hence of type 2 by Theorem \ref{thm-factors}.
\end{enumerate}
We recall that $\varphi_{e_1},\varphi_{e_2}$ and $\varphi_{e_3}$ denote the reduced weights of $\varphi$ on the von Neumann algebras $e_1 \cal{M} e_1,e_2 \cal{M} e_2$ and $e_3 \cal{M} e_3$. Note that theses weights are faithful normal positive functionals.

\begin{lemma}
\label{First-lemma} 
Let $(e_i)_{i \in I}$ be a family of projections of the the center $\Zc(A)$ of the $\JW$-algebra $A$ with sum $1$.
\begin{enumerate}
\item For any $x \in \cal{M}$ and any $i \in I$, we have $Q(e_i x e_i) \in e_i \cal{N}$.
\item For any $x \in \cal{M}$, we have 
\begin{equation}
\label{Decompose-Q}
Q(x)
=\sum_{i \in I}Q(e_i x e_i).
\end{equation}	

  \item If $i \in I$ and if $y$ belongs to $e_i \cal{N}$, we have $Q(y)=y$. 
	
	\item The projections $e_i$ belongs to the centralizer $\cal{M}^\varphi$ of $\varphi$.
\end{enumerate}
\end{lemma}

\begin{proof}
1. Since $e_i \in A$, we have $Q(e_i)=e_i$. By \cite[Theorem 2.2.2 (i)]{Sto13}, the selfadjoint element $e_i$ belongs to the definite set \eqref{Def-definite-set} of $Q$. For any $x \in \cal{M}$, we deduce by the well-known Broise's observation \cite[Lemma 4.1]{Sto80a} \cite[Proposition 2.1.7 (ii)]{Sto13} that 
\begin{equation}
\label{Divers-150}
Q(e_ixe_i)
=Q(e_i)Q(x)Q(e_i)
=e_iQ(x)e_i
=e_iQ(x)
\end{equation}
where we use in the last equality the fact that the element $e_i$ of $\Zc(A)$ commute with the elements of $\Ran Q=\cal{N}=A+\i A$. In particular, $Q(e_i x e_i)$ belongs to $e_i\cal{N}$. Similarly, we have $Q(e_i x e_i)=Q(x)e_i$. 

2. Now, for any $x \in \cal{M}$, we have
\begin{equation}
\label{Divers-151}
\sum_{i \in I} Q(e_i x e_i)
\ov{\eqref{Divers-150}}{=} \sum_{i \in I} e_iQ(x)
= \bigg( \sum_{i \in I} e_i\bigg)Q(x)
=Q(x).
\end{equation}

3. For any $x \in \cal{M}$, replacing $x$ by $e_i x$ in \eqref{Divers-151}, we deduce that 
\begin{equation}
\label{Divers-152}
Q(e_i x) 
\ov{\eqref{Divers-151}}{=} \sum_{j \in I} Q(e_j e_i x e_j)
=Q(e_ixe_i)
\end{equation}
and similarly $Q(x e_i)=Q(e_i x e_i)$. If $y \in e_i \cal{N}$, there exists $x \in \cal{N}$ such that $y=e_i x$. So we have 
$$
Q(y)
=Q(e_i x)
\ov{\eqref{Divers-152}}{=} Q(e_ixe_i)
\ov{\eqref{Divers-150}}{=} e_iQ(x)
=e_i x
=y.
$$

4. For any $x \in \cal{M}$, using the preservation of $\varphi$ in the first and the last equalities, we see that
\begin{equation*}
\label{}
\varphi(e_ix)
=\varphi(Q(e_i x))
\ov{\eqref{Divers-152}}{=} \varphi(Q(e_i x e_i))
=\varphi(Q(x e_i))
=\varphi(x e_i).
\end{equation*}
By \eqref{Charac-centralizer}, we conclude that $e_i$ belongs to the centralizer $\cal{M}^\varphi$.
%
\end{proof}


By the part 1 of this result, we deduce that we have a canonical isometric identification of each noncommutative $\L^p$-space $\L^p(e_i \cal{M}e_i)$ with the subspace $e_i\L^p(\mathcal{M})e_i$ of $\L^p(\cal{M})$. Note that each restriction $Q|e_i \cal{M} e_i \co e_i \cal{M} e_i \to e_i \mathcal{M} e_i$ is a faithful normal unital positive projection on $Ne_i$ which preserves the normal positive faithful linear form $\varphi|e_i \cal{M} e_i$. So to understand the projection $Q$, it suffices by Lemma \ref{First-lemma} to examine three cases separately.  

\vspace{0.2cm}


We need the following result for the two last cases. Here, we denote by $\L^{p,D}(\cal{M},\tau_0)$ the Dixmier noncommutative $\L^p$-space associated with a von Neumann algebra $M$ equipped with a normal semifinite faithful trace $\tau_0$. If $\tau_0$ is finite, recall that $\mathfrak{m}_{\tau_0}=\cal{M}$.

\begin{lemma}
\label{Lemma-identification}
Let $\cal{M}$ be a von Neumann algebra equipped with a normal semifinite faithful trace $\tau_0$. Let $T \co \cal{M} \to \cal{M}$ be a positive normal contraction which preserves $\tau_0$ and a normal linear form $\varphi=\tau_0(d\, \cdot)$ on $\cal{M}$ where $d \in \L^{1,D}(\cal{M},\tau)$. We denote by $h_{\varphi}$ the density operator of the weight $\varphi$. Suppose $1 \leq p < \infty$. Then the map $\mathfrak{m}_{\tau_0} \mapsto \L^p(\cal{M},\varphi)$, $x \mapsto h_{\varphi}^{\frac{1}{2p}} x h_{\varphi}^{\frac{1}{2p}}$ extends to a positive isometric map $\Phi \co \L^{p,D}(\cal{M},\tau_0) \to \L^p(\cal{M},\varphi)$ and if $T_{p,H} \co \L^p(\mathcal{M},\varphi) \to \L^p(\mathcal{M},\varphi)$ and $T_{p,D} \co \L^{p,D}(\cal{M},\tau_0) \to \L^{p,D}(\cal{M},\tau_0)$ denote the $\L^p$-extension of $T$ we have the following commutative diagram.
$$
\xymatrix @R=1cm @C=2cm{
\L^p(\cal{M},\varphi) \ar[r]^{T_{p,H}} &  \L^p(\cal{M},\varphi) \\
\L^{p,D}(\cal{M},\tau_0) \ar[u]^{\Phi}   \ar[r]_{T_{p,D}}   &  \L^{p,D}(\cal{M},\tau_0)\ar[u]_{\Phi} \\						
}
$$
\end{lemma}

\begin{proof}
First, we need to recall some information on the crossed product  $\cal{M} \rtimes_{\sigma^{\tau_0}} \R$ which is written in few lines in \cite[pp.~62-63]{Ter81}. Note that $\cal{M} \rtimes_{\sigma^{\tau_0}} \R$ is equal to the tensor product $\cal{M} \otvn \cal{A}$ where $\cal{A}$ is the von Neumann algebra generated by the translations $\lambda_s \co \L^2(\R) \to \L^2(\R)$ for $s \in \R$. With the Fourier transform $\cal{F} \co \L^2(\R) \to \L^2(\R)$ we can write $\cal{F}^{-1}\cal{A}\cal{F}=\L^\infty(\R)$ and 
\begin{equation}
\label{mult-translation}
\cal{F}^{-1}\lambda_s\cal{F}
=\e^{\i  s \cdot},\quad s \in \R.
\end{equation}
So we have a $*$-isomorphism $\Theta \co \cal{M} \otvn \cal{A} \to \cal{M} \otvn \L^\infty(\R)$, $y \ot z \mapsto y \ot \cal{F}^{-1} z \cal{F}$. As indicated in \cite[p.~62]{Ter81}, it is easy to check that the transformation $\theta_s \ov{\mathrm{def}}{=} \Theta^{-1} \circ \hat{\sigma}_s^{\tau_0} \circ \Theta$ of the dual action $\hat{\sigma}_s^{\tau_0} \co \cal{M} \otvn \cal{A} \to \cal{M} \otvn \cal{A}$ defined in \eqref{Dual-action} under $\Theta$ is determined by
\begin{equation}
\label{action-thetat}
\theta_s(x \ot f)
=x \ot \lambda_s(f), \quad s \in \R, x \in \cal{M}, f \in \L^\infty(\R).
\end{equation}
Recall that $\hat{\tau_0}$ denotes the dual weight on $\cal{M} \otvn \cal{A}$ defined in \eqref{Def-poids-dual}. Let $\tilde{\hat{\tau_0}} \ov{\mathrm{def}}{=} \hat{\tau_0} \circ \Theta^{-1}$ be the weight on $\cal{M} \otvn \L^\infty(\R)$ obtained by transport of structure. 
It is not difficult to prove that
\begin{equation}
\label{trace-identification}
\tilde{\hat{\tau_0}}
=\tau_0 \ot \int_{-\infty}^\infty.
\end{equation}
Let $\tau \ov{\mathrm{def}}{=} \tau_{\tau_0}$ be the canonical trace on the von Neumann algebra $\cal{M} \rtimes_{\sigma^{\tau_0}} \R=\cal{M} \otvn \cal{A}$. Set $\tilde{\tau} \ov{\mathrm{def}}{=} \tau \circ \Theta^{-1}$ be the normal semifinite faithful trace on $\cal{M} \otvn \L^\infty(\R)$ obtained by transport of structure. Then we have
\begin{equation}
\label{}
\big(\D \tilde{\hat{\tau_0}} : \D\tilde{\tau}\big)_t
=\Theta\big(\D\hat{\tau_0} : \D\tau)_t\big)
\ov{\eqref{trace-carac}}{=} \Theta(1 \ot \lambda_t)
=1 \ot \cal{F}^{-1}  \lambda_t \cal{F}
\ov{\eqref{mult-translation}}{=} 1 \ot \e^{\i t \cdot}
=(1 \ot \exp)^{\i t}, \quad t \in \R.
\end{equation}
Hence, by \eqref{Radon-Nikodym-1}, the Pedersen-Takesaki derivative $\displaystyle \frac{\d\tilde{\hat{\tau_0}}}{\d\tilde{\tau}}$ is equal to $1 \ot \exp$. 
By \cite[p.~62]{Str81}, it follows that
\begin{equation}
\label{Identification-tilde-tau}
\tilde{\tau}
=\tilde{\hat{\tau_0}}\big(\cdot (1 \ot \exp)^{-1}\big)
\ov{\eqref{trace-identification}}{=} \tau_0 \ot \e^{-s} \d s. 
\end{equation}
Moreover, recall that $h_{\tau_0} = \displaystyle \frac{\d\hat{\tau_0}}{\d\tau}$. By transport of structure, we have 
\begin{equation}
\label{inter-12467}
1 \ot \exp
=\Theta(h_{\tau_0}).
\end{equation}
The isomorphism $\Theta \co \cal{M} \otvn \cal{A} \to \cal{M} \otvn \L^\infty(\R)$ extends to an isomorphism $\Theta \co \L^0(\cal{M} \otvn \cal{A},\tau) \to \L^0(\cal{M} \otvn \L^\infty(\R),\tilde{\tau})$ which allows to identify isometrically the Haagerup noncommutative $\L^p$-space $\L^{p,H}(\cal{M},\tau_0)$ with the range $\Theta(\L^p(\cal{M},\tau_0))$. 
By \cite[pp.~62-63]{Ter81}, we have an order and isometric isomorphism $\eta \co \L^{p,D}(\cal{M},\tau_0) \to \Theta(\L^p(\cal{M},\tau_0))$, $x \mapsto x \ot \exp^{\frac{1}{p}}$. From \eqref{kappa}, we also have a canonical map $\kappa$ which induces an order and isometric identification $\kappa \co \L^p(\cal{M},\tau_0) \to \L^p(\cal{M},\varphi)$ for all $p$ such that 
\begin{equation}
\label{h-tau-0}
\kappa(h_{\tau_0})
=h_\varphi.
\end{equation}
Now, for any $x \in \mathfrak{m}_{\tau_0}$ we can compute 
\begin{align*}
\MoveEqLeft
\kappa \circ \Theta^{-1} \circ \eta(x)          
=\kappa \circ \Theta^{-1}\big(x \ot \exp^{\frac{1}{p}}\big) 
=\kappa \circ \Theta^{-1}\big((1 \ot \exp^{\frac{1}{2p}})(x \ot 1)(1 \ot \exp^{\frac{1}{2p}})\big) \\
&=\kappa \big(\Theta^{-1}(1 \ot \exp)^{\frac{1}{2p}}\Theta^{-1}(x \ot 1)\Theta^{-1}(1 \ot \exp)^{\frac{1}{2p}}\big) \\
&\ov{\eqref{inter-12467}}{=}\kappa \big(h_{\tau_0}^{\frac{1}{2p}}x h_{\tau_0}^{\frac{1}{2p}}\big) 
= \kappa(h_{\tau_0})^{\frac{1}{2p}}x\kappa(h_{\tau_0})^{\frac{1}{2p}}
\ov{\eqref{h-tau-0}}{=}h_\varphi^{\frac{1}{2p}} x h_\varphi^{\frac{1}{2p}}
=\Phi(x).           
\end{align*} 
By composition, we conclude that $\Phi$ induces an order isometric isomorphism from $\L^{p,D}(\cal{M},\tau_0)$ into $\L^p(\cal{M},\varphi)$. If $x \in \mathfrak{m}_{\tau_0}$, we have 
\begin{align*}
\MoveEqLeft
T_{p,H} \circ \Phi(x)
=T_{p,H}\big(h_\varphi^{\frac{1}{2p}}xh_\varphi^{\frac{1}{2p}}\big)
\ov{\eqref{Map-extension-Lp}}{=} h_\varphi^{\frac{1}{2p}} T(x) h_\varphi^{\frac{1}{2p}}            
=\Phi(T(x))
=\Phi \circ T_{p,D}(x).
\end{align*} 
The proof is complete.
%
\end{proof}

\paragraph{Case 1} 
We consider a normal positive faithful linear form $\varphi$ on a von Neumann algebra $\cal{M}$. We suppose that $Q \co \cal{M} \to \cal{M}$ is a $\varphi$-preserving normal unital faithful positive projection on a $\JW^*$-subalgebra $\cal{N}$ such that the associated $\JW$-algebra $A = \cal{N}_\sa$ is the selfadjoint part of the von Neumann algebra $A''$. Note that $\cal{N}=A +\i A=(A'')_\sa+\i (A'')_\sa=A''$. Consequently, the range $\cal{N}$ of $Q$ is a von Neumann algebra. By Tomiyama's theorem \cite[p.~117]{Str81}, we conclude that $Q \co \cal{M} \to \cal{M}$ is a normal faithful conditional expectation on $N$. Since $\varphi$ is preserved, we can consider by \eqref{Map-extension-Lp} its $\L^p$-extension $Q_p \co \L^p(\cal{M}) \to \L^p(\cal{M})$. Its range is isometrically isomorphic to a noncommutative $\L^p$-space.

\vspace{0.2cm}

\paragraph{Case 2} 
We consider a normal positive faithful linear form $\varphi$ on a von Neumann algebra $\cal{M}$. We suppose that $Q \co \cal{M} \to \cal{M}$ is a $\varphi$-preserving normal unital faithful positive projection on a $\JW^*$-subalgebra $\cal{N}$ such that the associated $\JW$-algebra $A = \cal{N}_\sa$ is purely real. We only look the case where $A$ is a $\JW$-factor which is not isomorphic to the selfadjoint part of a von Neumann algebra and where $\cal{M}$ is finite. 

\begin{prop}
\label{Prop-analysis-good-centers}
Let $\cal{M}$ be a finite von Neumann algebra equipped with a faithful normal positive linear form $\varphi$. Let $Q \co \cal{M} \to \cal{M}$ be a $\varphi$-preserving faithful normal unital positive projection on a $\JW^*$-subalgebra $\cal{N}$ such that the associated $\JW$-algebra $A \ov{\mathrm{def}}{=} \cal{N}_\sa$ is a purely real factor which is not isomorphic to the selfadjoint part of a von Neumann algebra. Suppose $1 < p <\infty$. Then the range of the $\L^p$-extension $Q_p \co \L^p(\cal{M}) \to \L^p(\cal{M})$, $h_{\varphi}^{\frac{1}{2p}}xh_{\varphi}^{\frac{1}{2p}} \mapsto h_{\varphi}^{\frac{1}{2p}} P(x) h_{\varphi}^{\frac{1}{2p}}$  
is isometrically isomorphic to an interpolation space of the form $(\cal{N},\cal{N}_*)_{\frac{1}{p}}$.
\end{prop}

\begin{proof}
Note that $A \subset \cal{N} \subset A''$. So, in this paragraph, we can consider $Q$ as a map $Q \co \cal{M} \to A''$. We consider the normal faithful positive linear form $\varphi'' \ov{\mathrm{def}}{=} \varphi|A''$. Note that by Proposition \ref{Prop-center}, we have $\Zc(A)=(\Zc(A''))_\sa$. By Theorem \ref{Th-HaS95-3.2}, there exists a (unique) faithful normal conditional expectation $\E \co \cal{M} \to A''$ on the von Neumann algebra $A''$ such that $Q = P_\can \circ \E$ where the canonical projection $P_\can \co A'' \to A''$ on $\cal{M}$ is defined in Section \ref{Sec-Jordan}. We deduce that $Q|A''=P_{\can}$. 
We infer that $P_\can$ preserves the linear form $\varphi''$. Now, we have
$$
\varphi
=\varphi'' \circ Q
=\varphi'' \circ P_\can \circ \E
=\varphi'' \circ \E.
$$
From Takesaki's theorem \cite[Theorem 10.1 p.~130]{Str81}, we deduce that $\varphi''$ is semifinite and that the von Neumann algebra $A''$ is invariant under the modular group, i.e.
$$
\sigma_t^{\varphi''}(A'')
=A'', \quad t \in \R.
$$
We conclude that the noncommutative $\L^p$-space $\L^p(A'',\varphi'')$ can be naturally isometrically identified with a subspace of $\L^p(\cal{M})$. So it suffices to understand the range of the $\L^p$-extension of $P_{\can}$.

By Proposition \ref{prop-Ayu151}, the von Neumann algebra $A''$ is a factor. Since $\cal{M}$ is finite, the von Neumann subalgebra $A''$ is finite. By \cite[Theorem 8.2.8]{KaR97b}, we deduce that $A''$ admits a unique normalized finite trace $\tau$. Consider the $*$-antiautomorphism $\alpha \co A'' \to A''$ defined in Section \ref{Sec-Jordan}. Since $\tau \circ \alpha$ is also a normalized finite trace on $A''$, we conclude that $\alpha$ is trace preserving, that is $\tau \circ \alpha=\tau$. By linearity, it is immediate that the normal unital positive projection $P_\can \co A'' \to A''$ is also trace preserving. By Proposition \ref{Prop-selfadjoint} and Proposition \ref{Prop-Jordan-se=se}, note that $P_\can$ is selfadjoint with respect to $\tau$.

The (faithful) normal positive linear form $\varphi''$ can be written $\varphi''=\tau(d \,\cdot)$ where $d$ is an element of the Dixmier noncommutative $\L^1$-space $\L^{1,D}(A'',\tau)$. We denote by $h_{\varphi''}$ the density operator associated with $\varphi''$.  Now by Lemma \ref{Lemma-identification}, the linear map $A'' \mapsto \L^p(A'',\varphi'')$, $x \mapsto h_{\varphi''}^{\frac{1}{2p}} x h_{\varphi''}^{\frac{1}{2p}}$ extends to a positive isometric map $\Phi \co \L^{p,D}(A'',\tau) \to \L^p(A'',\varphi'')$ and we have the following commutative diagram.
$$
\xymatrix @R=1cm @C=2cm{
\L^p(A'',\varphi'') \ar[r]^{P_{\can,p}} &  \L^p(A'',\varphi'') \\
\L^{p,D}(A'',\tau) \ar[u]^{\Phi} \ar[r]_{P_{\can,p,D}}   &  \L^{p,D}(A'',\tau) \ar[u]_{\Phi} \\							
}
$$
Since $P_\can$ is selfadjoint, we know by \cite[p.~43]{JMX06} that $P_{\can,1,D} \co \L^{1,D}(A'') \to \L^{1,D}(A'')$ identifies to the preadjoint $(P_\can)_* \co (A'')_* \to (A'')_*$. By Proposition \ref{Prop-Fabian}, the range of the latter map is isometric to the predual $(\Ran P_\can)_*=\cal{N}_*$ of the range of $P_\can$. With Lemma \ref{Lemma-interpolation}, we deduce that the range of $P_{\can,p,D}$ is isometric to some interpolation space $(\cal{N},\cal{N}_*)_{\frac{1}{p}}$. Using the commutative diagram, we deduce the same thing for $P_{\can,p}$. The proof is complete.
\end{proof}


\begin{remark} \normalfont
It seems to the author that we can use direct theory to generalize to the case of a reversible purely real $\JW^*$-algebra $\cal{N}$ such that $A$ does not have direct summands isomorphic to the selfadjoint part of a von Neumann algebra. Probably, with the first case 1, we can obtain the case of a reversible purely real $\JW^*$-algebra $\cal{N}$.
\end{remark}

\paragraph{Case 3} We consider a normal positive faithful linear form $\varphi$ on a von Neumann algebra $\cal{M}$. We suppose that $Q \co \cal{M} \to \cal{M}$ is a $\varphi$-preserving normal unital faithful positive projection on a $\JW^*$-subalgebra $\cal{N}$ such that the associated $\JW$-algebra $A = \cal{N}_\sa$ is of type $\I_2$. We only look the case where $A$ generates the von Neumann algebra $\cal{M}$ (and where $\cal{M}$ has separable predual). It seems to the author that it is possible to obtain a slightly more general statement with \cite[Lemma 2.2]{HaS95} but we do not know if this result allows to remove definitively this probably unnecessary assumption.

\begin{prop}
\label{Prop-analysis-I2}
Let $\cal{M}$ be a von Neumann algebra with separable predual equipped with a faithful normal positive linear form $\varphi$. Let $Q \co \cal{M} \to \cal{M}$ be a faithful normal unital positive projection which preserves $\varphi$ on a $\JW^*$-algebra $\cal{N}$ such that the associated $\JW$-algebra $A \ov{\mathrm{def}}{=} \cal{N}_\sa$ is of type $\I_2$. Suppose that $A''=\cal{M}$. Then the range of the $\L^p$-extension $Q_p \co \L^p(\cal{M}) \to \L^p(\cal{M})$, $h_{\varphi}^{\frac{1}{2p}}xh_{\varphi}^{\frac{1}{2p}} \mapsto h_{\varphi}^{\frac{1}{2p}} P(x) h_{\varphi}^{\frac{1}{2p}}$ is isometrically isomorphic to an interpolation space of the form $(\cal{N},\cal{N}_*)_{\frac{1}{p}}$.
\end{prop}


\begin{proof}
By Example \ref{ex-type-II-JW}, there exist an index set $I$, a family $(\Omega_i)_{i \in I}$ of second countable locally compact spaces, a family of Radon measures $(\mu_i)_{i \in I}$ on the spaces $\Omega_i$ and a family $(S_i)_{i \in I}$ of spin factors, each of dimension at most countable (and strictly greater than 1) giving an isomorphism 
$$
A
=\oplus_{i \in I} \L^\infty_\R(\Omega_i,\mu_i,S_i).
$$ 
We need to take account of the generated von Neumann algebra $\cal{M}=A''$. So it is clear that we have a $*$-isomorphism $A''=\oplus_{i \in I} \L^\infty(\Omega_i,\mu_i,S_i'')$ for some \textit{concrete} spin factor $S_i$. By Theorem \ref{th-proj-finite}, the von Neumann algebra $A''$ is necessarily \textit{finite}, hence by \cite[Proposition 6.3.1 3)]{Li92} each summand of $A''$ is also finite. Note that by \cite[p.~68]{Sak98} we have a $*$-isomorphism $\L^\infty(\Omega_i,\mu_i,S_i'')=\L^\infty(\Omega_i,\mu_i) \otvn S_i''$. By \cite[Proposition 2.6.1]{Sak98}, we conclude that each von Neumann algebra $S_i''$ is finite. 

If $\cal{A}$ is the CAR algebra over the complex Hilbert space $\ell^2$, recall that by \cite[Theorem 6.2.2]{HOS84} the $\mathrm{C}^*$-algebra generated by the spin factor $S_i$ is $*$-isomorphic to
\begin{equation}
\label{}
\begin{cases}
\M_{2^{n-1}} \oplus \M_{2^{n-1}} & \text{ if } \dim S_i=2n \\
\M_{2^n} & \text{ if } \dim S_i=2n+1 \\
\cal{A} &\text{ if } \dim S_i= \infty
\end{cases}.
\end{equation}
In the case where $\dim S_i = \infty$, by \cite[Proposition 12.1.3]{KaR97b} and since $A''$ \textit{is finite}\footnote{\thefootnote. We warn the reader that the von Neumann algebra generated by an infinite dimensional spin factor is not necessarily a factor of type $\II_1$.}, the von Neumann algebra $S_i''$ is a factor of type $\II_1$ which is of course hyperfinite by definition \cite[p.~895]{KaR97b}. By \cite[Proposition 12.1.4]{KaR97b}, we can suppose that the spin factor $S_i$ is canonically embedding in the unique hyperfinite factor $\cal{R}$ of type $\II_1$ with separable predual. 

Now, it is easy to check that we have a $*$-isomorphism $A''=\oplus_i\L^\infty(\Omega_i,\mu_i,S_i'')$ where
\begin{equation}
\label{}
S_i''
\ov{\mathrm{def}}{=}
\begin{cases}
\M_{2^{n-1}} \oplus \M_{2^{n-1}} & \text{ if } \dim S_i=2n \\
\M_{2^n} & \text{ if } \dim S_i=2n+1 \\
\cal{R} &\text{ if } \dim S_i= \infty
\end{cases}
\end{equation}
and that with Lemma \ref{First-lemma} we can reduce the problem to the case of a faithful normal unital positive projection map 
$$
Q_i \co \L^\infty(\Omega_i,S_i'') \to \L^\infty(\Omega_i,S_i'')
$$ 
on $\L^\infty(\Omega_i,(S_i)_\C)$ which preserves a faithful normal positive linear form $\varphi_i$ where $(S_i)_\C$ is the $\JW^*$-algebra associated to $S_i$.

Note that by \cite[Corollary p.~178]{Dix81} we have a $*$-isomorphism $\L^\infty(\Omega_i,S_i'')=\int_{\Omega_i}^\oplus S_i'' \d\mu_i(\omega)$ (we can suppose that the support of $\mu_i$ is $\Omega_i$). With \cite[Lemma 14.1.19]{KaR97a} and \cite[Proposition 8.34 p.~285]{Tak02}), 
we can decompose the form $\varphi_i$ as a direct integral $\varphi_i =\int_{\Omega_i}^\oplus \varphi_{i,\omega} \d \mu_i(\omega)$ of faithful normal positive linear forms $\varphi_{i,\omega}$ on the von Neumann algebra $S_i''$. We denote by 
$$
Q_{i,\sa} \co \L^\infty_\R(\Omega_i,(S_i'')_\sa) \to \L^\infty_\R(\Omega_i,(S_i'')_\sa)
$$ 
the restriction of $Q_i$ on the $\JW$-algebra $\L^\infty(\Omega_i,S_i'')_\sa=\L^\infty_\R(\Omega_i,(S_i'')_\sa)$ which is a projection on the $\JW$-algebra $\L^\infty_\R(\Omega_i,S_i)$. We can decompose $Q_{i,\sa}$ 
as a direct integral $Q_{i,\sa}=\int_{\Omega_i}^{\oplus} Q_{i,\sa,\omega} \d \mu_i(\omega)$ of faithful normal projections $Q_{i,\sa,\omega} \co (S_i'')_\sa \to (S_i'')_\sa$ on $S_i$ preserving the restriction of $\varphi_{i,\omega}$ on $(S_i'')_\sa$. By \cite[Lemma 4.4.13]{HOS84}, these maps are Jordan conditional expectations. By complexification, we deduce that $Q_i=\int_{\Omega_i}^{\oplus} Q_{i,\sa,\omega,\C} \d\mu_i(\omega)$ where $Q_{i,\sa,\omega,\C} \co S_i'' \to S_i''$. By Lemma \ref{Lemma-complexification-conditional-expectation}, note that $Q_{i,\sa,\omega,\C}$ is a Jordan conditional expectation onto $S_{i,\C}$.


We equip the spin factor $S_i$ with its unique tracial state $\tau_{S_i}$, see Example \ref{ex-trace-spin}. If $(\tau_{S_i})_\C$ denote its complexification on its  associated $\JW^*$-algebra $(S_i)_\C$ then by \cite[Lemma 2.2]{HaS95}, the functional $\tau_i \ov{\mathrm{def}}{=} (\tau_{S_i})_\C \circ Q_{i,\sa,\omega,\C}$\footnote{\thefootnote. Here, we consider $Q_{i,\sa,\omega,\C}$ as a map $Q_{i,\sa,\omega,\C} \co S_i'' \to (S_i)_\C$.} is a normalized normal faithful trace on the von Neumann algebra $S_i''$. Note that
$$
\tau_i \circ Q_{i,\sa,\omega,\C}
=(\tau_{S_i})_\C \circ Q_{i,\sa,\omega,\C} \circ Q_{i,\sa,\omega,\C}
=(\tau_{S_i})_\C \circ Q_{i,\sa,\omega,\C}
=\tau_i,
$$
i.e. $Q_{i,\sa,\omega,\C}$ preserves $\tau_i$. We infer that $Q_i=\int_{\Omega_i}^{\oplus} Q_{i,\sa,\omega,\C} \d\mu_i(\omega)$ preserves the trace $\int_{\Omega_i} \ot \tau_i$. By Proposition \ref{Prop-selfadjoint} and Proposition \ref{Prop-Jordan-se=se}, note that $Q_i$ is selfadjoint with respect to the normal semifinite faithful trace $\int_{\Omega_i} \ot \tau$.

Each (faithful) normal positive linear form $\varphi_{i,\omega}$ can be written $\tau_i(d_{i,\omega}\, \cdot)$ where $d_{i,\omega}$ is an element of the Dixmier noncommutative $\L^1$-space $\L^{1,D}(S_i'',\tau_i)$. Introducing $d_i \ov{\mathrm{def}}{=}\int_{\Omega_i}^\oplus d_{i,\omega} \d \mu_i(\omega)$, we have
\begin{align*}
\MoveEqLeft
\varphi_i
=\int_{\Omega_i}^\oplus \varphi_{i,\omega} \d \mu_i(\omega)
=\int_{\Omega_i}^\oplus \tau_i(d_{i,\omega} \cdot) \d \mu_i(\omega) \\
&=\bigg(\int_{\Omega_i} \ot \tau_i\bigg)\bigg(\int_{\Omega_i}^\oplus d_{i,\omega} \d \mu_i(\omega)\, \cdot \bigg)
=\bigg(\int_{\Omega_i} \ot \tau_i\bigg)(d_i \, \cdot).            
\end{align*}
We denote by $h_{\varphi_i}$ the density operator associated with $\varphi_i$. Now by Lemma \ref{Lemma-identification}, we have a positive isometric map $\Phi \co \L^p(\Omega_i,\mu_i,\L^{p,D}(S_i'',\tau_i)) \to \L^p(\L^\infty(\Omega,\mu_i,S_i''),\varphi_i)$, $x \mapsto h_{\varphi_i}^{\frac{1}{2p}} x h_{\varphi_i}^{\frac{1}{2p}}$ and we have the following commutative diagram.
$$
\xymatrix @R=1cm @C=2cm{
\L^p(\L^\infty(\Omega_i,\mu_i,S_i'')),\varphi_i) \ar[r]^{Q_{i,p}} &  \L^p(\L^\infty(\Omega_i,\mu_i,S_i'')),\varphi_i) \\
\L^p(\Omega_i,\mu_i,\L^{p,D}(S_i'',\tau_i)) \ar[u]^{\Phi} \ar[r]_{Q_{i,p,D}}   &   \L^p(\Omega_i,\mu_i,\L^{p,D}(S_i'',\tau_i)) \ar[u]_{\Phi} \\					
}
$$
Since $Q_i$ is selfadjoint, we know by \cite[p.~43]{JMX06} that the map $Q_{i,1,D} \co \L^1(\Omega_i,\mu_i,\L^{1,D}(S_i'',\tau_i)) \to \L^1(\Omega_i,\mu_i,\L^{1,D}(S_i'',\tau_i))$ identifies to the preadjoint $(Q_i)_* \co (\L^\infty(\Omega_i,\mu_i,S_i''))_* \to (\L^\infty(\Omega_i,\mu_i,S_i''))_*$. By Proposition \ref{Prop-Fabian}, the range of the latter map is isometric to the predual $(\Ran Q_i)_*=\L^\infty(\Omega_i,\mu_i,S_{i,\C})_*$ of the range of $Q_i$. With Lemma \ref{Lemma-interpolation}, we deduce that the range of $Q_{i,p,D}$ is isometric to some interpolation space $(\L^\infty(\Omega_i,\mu_i,S_{i,\C}),\L^\infty(\Omega_i,\mu_i,S_{i,\C})_*)_{\frac{1}{p}}$. Using the commutative diagram, we deduce the same thing for $Q_{i,p}$. The proof is complete.
\end{proof}

\section{Open questions}
\label{sec-open-questions}

It is natural to state the following conjecture in view of our results. In a next version of this preprint, we will hope finish the case where $\cal{M}$ is a finite von Neumann algebra.

\begin{conj}
\label{conj1}
Suppose $1 < p < \infty$ with $p \not=2$. Let $X$ be a Banach space. Then $X$ is isometric to a positively contractively complemented subspace of a Haageup noncommutative $\L^p$-space $\L^{p,H}(\cal{M},\varphi)$ where $\cal{M}$ is a $\sigma$-finite von Neumann algebra equipped with a normal faithful state $\varphi$ if and only if $X$ is isometric to a nonassociative $\L^p$-space associated with a $\JW^*$-algebra equipped with a normal faithful state.
\end{conj}



%
%



\vspace{0.2cm}

\textbf{Acknowledgment}.
The author acknowledges support by the grant ANR-18-CE40-0021 (project HASCON) of the French National Research Agency ANR. Finally, I would like to thank David Blecher, Bruno Iochum, Yves Raynaud, Eric Ricard for useful very short discussions and Quanhua Xu for providing me an expanded version of his paper \cite{HJX10} from his \textit{own} initiative. Finally, the author would like to thank Miguel Cabrera Garc{\'i}a and {\'A}ngel Rodr{\'\i}guez Palacios for providing some information on $\JBW^*$-algebras and $\JBW^*$-triples.


{\footnotesize


\noindent C\'edric Arhancet\\ 
\noindent 13 rue Didier Daurat, 81000 Albi, France\\
URL: \href{http://sites.google.com/site/cedricarhancet}{https://sites.google.com/site/cedricarhancet}\\
cedric.arhancet@protonmail.com\\

}

\end{document}